\documentclass[10pt,leqno, final]{amsart}

\usepackage{amsfonts,amssymb}
\usepackage{color}
\usepackage{graphicx}
\usepackage{mathtools}
\usepackage{extarrows}
\usepackage{algorithmic}   
\usepackage{caption}
\usepackage{subcaption}
\usepackage{tikz}
\usepackage{mathrsfs}
\usepackage[isbn=false]{biblatex}

\allowdisplaybreaks

\usepackage{todonotes}

\usepackage[colorlinks,linkcolor=blue,citecolor=blue,urlcolor=blue]{hyperref}
\usepackage[nameinlink,noabbrev,capitalize]{cleveref}

\newtheorem{theorem}{Theorem}[section]
\newtheorem{lemma}[theorem]{Lemma}
\newtheorem{corollary}[theorem]{Corollary}
\newtheorem{proposition}[theorem]{Proposition}
\newtheorem{assumption}[theorem]{Assumption}

\theoremstyle{definition}
\newtheorem{definition}[theorem]{Definition}
\newtheorem{example}[theorem]{Example}
\newtheorem*{example*}{Example}
\newtheorem{remark}[theorem]{Remark}

\numberwithin{equation}{section}

\usepackage{my_latex_commands}

\newcommand{\lip}{\mathrm{Lip}}
\newcommand{\wdist}{\mathrm{D}}
\newcommand{\yp}{\mathfrak{y}}
\renewcommand{\wp}{\mathfrak{w}}
\newcommand{\mres}{\mathbin{\vrule height 1.6ex depth 0pt width
		0.13ex\vrule height 0.13ex depth 0pt width 1.3ex}}
\newcommand{\mmres}{\mathbin{\vrule height 1.0ex depth 0pt width
		0.13ex\vrule height 0.13ex depth 0pt width 0.7ex}}

\usepackage[normalem]{ulem}

\newcommand\cmchange[1]{\textcolor{black}{#1}}

\title[Optimal control with transport regularization]
{Optimal control of the Poisson equation with transport regularization: Properties of optimal transport plans and transport maps}

\author{Christian Meyer} \address{Technische Universit\"at Dortmund, Fakult\"at f\"ur
  Mathematik, Lehrstuhl LSX, Vogelpothsweg 87, 44227 Dortmund, Germany}
\email{christian2.meyer@tu-dortmund.de}

\author{Gerd Wachsmuth} \address{BTU Cottbus-Senftenberg, Fakult\"at 1, 
Fachgebiet Optimale Steuerung, Platz der Deutschen Einheit 1, 03046 Cottbus, Germany}
\email{gerd.wachsmuth@b-tu.de}

\begin{document}

\date{\today} 

\begin{abstract} 
	An optimal control problem in the space of Borel measures governed by the Poisson equation is investigated. 
    The characteristic feature of the problem under consideration is the Tikhonov regularization term in form of the transportation distance of the control
    to a given prior. Existence of optimal solutions is shown and first-order necessary optimality conditions are derived. 
    The latter are used to deduce structural a priori information about the optimal control and its support based on 
    properties of the associated optimal transport plan.
\end{abstract}

\maketitle

\section{Introduction}

We consider an optimal control governed by the Poisson equation, where the control space is the space of regular Borel measures. 
In order to ensure this regularity of the control, we do not add the total variation of the control, also known as Radon norm, as Tikhonov regularization to the objective, 
which is frequently done in the literature, see, e.g., \cite{ClasonKunisch2011}. Instead, we consider the \emph{transportation distance to a given prior} as regularizer, which leads to the
optimal control problem
\begin{equation}\tag{P}\label{eq:optconintro}
    \left\{\quad
    \begin{aligned}
        \min \quad & J(y) + \alpha\, \wdist^c_{u^0}(u)\\
        \text{w.r.t.} \quad & y \in W_0^{1,q}(\Omega), u \in \frakM(\omega_1), \\
        \text{s.t.} \quad & - \laplace y = u \text{ in } \Omega, \quad y = 0 \text{ on } \partial\Omega .
    \end{aligned}
    \right.
\end{equation}
Herein, $\Omega \subset \R^{d_1}$, $d_1 \in \set{2, 3}$, is a bounded Lipschitz domain and $\omega_0 \subset \R^{d_0}$ and $\omega_1 \subset \overline{\Omega}$ are given compact sets.
Moreover, $J$ is a given objective, $\alpha > 0$, and $\wdist^c_{u^0}(u)$ denotes the 
transportation distance of the control $u$ to the given prior $u^0 \in \frakM(\omega_0)$.
If (in case $d_0 = d_1$) the transportation costs $c$ equal the $\gamma$-th power of the 
Euclidean distance with $\gamma \geq 1$, then $\wdist^c_{u^0}(u)$ is simply the Wasserstein distance of $u$ and $u^0$ of order $\gamma$
raised to the $\gamma$-th power,
see \cite[Section~6]{Vil09} for details. For the precise definition of $\wdist^c_{u^0}(u)$ as well as for the precise assumptions on the data, we refer to \cref{sec:setting}. 

A motivation for considering the Wasserstein distance instead of the distance measured in the Radon norm, i.e., $\|u - u^0\|_{\frakM(\omega_0)} = |u - u^0|(\omega_0)$, is 
its advantageous continuity properties. To be more precise, if $\omega_0 = \omega_1$ and $\omega_0$ is locally compact (and thus Polish),
then \cite[Theorem~6.9]{Vil09} implies that, if a sequence $\{u_k\} \subset \frakM(\omega_0)$ of non-negative Borel measures with $u_k(\omega_0) = u^0(\omega_0)$ for all $k \in \N$ 
converges weakly$^\ast$ to $u^0$, then the Wasserstein distance between $u_k$ and $u^0$ converges to zero, whereas $\|u_k - u^0\|_{\frakM(\omega_0)}$ does in general not
(consider for instance a sequence of Diracs $\delta_{x_k}$ with $x_k \to x_0$). Therefore, especially in context of inverse problems, it might make sense to consider the 
Wasserstein distance or, more generally, a transportation distance as regularizer.

A transportation distance as Tikhonov regularization has rarely been investigated so far in the context of optimal control. 
In \cite{BorchardWachsmuth2025:1}, Brenier's celebrated theorem has been used to reformulate a problem of type \eqref{eq:optconintro} 
with a prior $u^0$ that is absolutely continuous w.r.t.\ the Lebesgue measure in terms of an optimization problem 
over the set of convex functions. Under the additional assumption that $\omega_1$ just consists of finitely many points, a semismooth Newton method is 
derived to solve the problem.
Other optimal control problems involving the Wasserstein distance have been considered in the literature, 
see for instance \cite{BonnetFrankowska2021} and the references therein, 
but there the state and not the control function is an element of the Wasserstein space.
With regard to inverse problems, Wasserstein regularization has been considered for time-dependent problems in 
\cite{BrediesFanzon2020, mauritz2024bayesian}. The so-called Kantorovich-Rubinstein norm as regularizer is investigated in 
\cite{CarioniIglesiasWalter2023}.

In contrast to the transport distance, the Radon norm as Tikhonov regularization has been considered in numerous contributions. We only mention
\cite{ClasonKunisch2011, CasasKunisch2014, CasasKunisch2019:1, PieperWalter2021} and the references therein. 
Special emphasis has been laid on the sparsity pattern of the optimal control as a consequence of 
the first-order necessary optimality conditions involving the subdifferential of the Radon norm, 
see \cite{ClasonKunisch2011, CasasZuazua2013, BrediesPikkarainen2013, CasasClasonKunisch2013:1, 
CasasKunisch2014, KunischPieperVexler2014, CasasVexlerZuazua2015, CasasKunisch2016:1, 
KunischTrautmannVexler2016, TrautmannVexlerZlotnik2018:1, CasasKunisch2019:2, LeykekhmanVexlerWalter2020}. 
Here we pursue a similar goal and aim to derive characteristic features of the optimal control from the optimality system 
associated with \eqref{eq:optconintro}. There are in principle two ways to deduce structural a priori information about the optimal control and its support:
On the one hand, one can look at the mass of the prior located at a point in $\omega_0$ and investigate where it is transported to. On the other hand, 
one can analyze where the mass of the optimal control a point in $\omega_1$ is coming from. We will employ both perspectives in different scenarios, 
depending on the transportation costs $c$ and the objective $J$. 

The paper is organized as follows: After presenting our standing assumptions and the precise formulation of \eqref{eq:optconintro}
in the following section, we will turn to the existence of optimal solutions in \cref{sec:existence}. Afterwards, \cref{sec:fon} is dedicated to 
the derivation of an optimality system as first-order necessary optimality condition, which is a straightforward consequence of adjoint calculus along 
with the characterization of the subdifferential of the transportation distance that is provided in \cref{sec:subdifftrans}.
We then consider the case where $J$ is of tracking type over the entire domain $\Omega$ in \cref{sec:tracking}. 
As it turns out, one can show that the optimal state is essentially bounded in this case which gives in turn that the optimal control 
has no atom in $\interior(\omega_1)$. The underlying analysis is based on the second perspective, i.e., we investigate where the mass of the optimal control 
is coming from. The first perspective is then taken in \cref{sec:stritctconv}, where strictly convex transportation costs are considered 
under an additional curvature assumption on the adjoint state. This allows to show that the mass of the prior at a single point in $\omega_0$ 
is transported to just one point in $\omega_1$ so that sparsity patterns of the prior carry over to the optimal control. 
Passing on to \cref{sec:abscont}, we change the perspective again, i.e., we study where the mass is coming from.
In this way, we show that, in case of power-type transportation costs, if the prior 
is absolutely continuous w.r.t.\ the Lebesgue measure, then the same holds for the optimal control, too, at least in the interior of $\Omega$.
The last \cref{sec:metric} is dedicated to the case with metric transportation costs, where $c$ is just the Euclidean distance of two points. 
While the first perspective allows to prove a sparsity-type result, provided that the Lipschitz constant of the adjoint state is sufficiently small, 
the second allows us to show that the optimal control is absolutely continuous w.r.t.\ the $(d_1-1)$-dimensional Hausdorff measure 
in the interior of $\omega_1$, provided that the prior is again absolutely continuous w.r.t.\ the Lebesgue measure.

\section{Precise setting and standing assumptions}\label{sec:setting}

Let us shortly introduce the notation used throughout the paper. Given a point in $x \in \R^d$, its Euclidean norm is denoted by $\|x\|$ 
and, for the Euclidean inner product, we write $a \cdot b = \dual{a}{b}$, $a, b\in \R^d$.
By $B_r(x)$ we denote the open ball of radius $r>0$ around $x\in \R^d$. 
Given two points $x, \xi \in \R^d$, we define the segment $[x, \xi] \coloneqq \{ (1-t)\, x + t \, \xi \colon t \in [0,1]\}$. 
The Borel-$\sigma$-algebra associated with an open set $\Omega \subset \R^d$ is denoted by $\BB(\Omega)$. 
By $\frakM(\Omega;\R^m)$ we denote the space of vector-valued Borel measures on $\Omega$. If $m = 1$, we simply write $\frakM(\Omega)$.
Note that all vector-valued Borel measures on $\Omega$ are regular, see \cite[Theorem~2.18]{Rudin1987}.
The Lebesgue measure on $\R^d$ 
is $\lambda^d$ and, to ease notation, its restriction to an open set $\Omega \subset \R^d$ is frequently denoted by the same symbol. 
    Recall that the support of a (nonnegative) Borel measure $\mu$ on a compact set $X \subset \R^d$
    is defined via
    \begin{equation*}
        \supp(\mu)
        :=
        \set{
            x \in X \given \forall r > 0 : \mu(B_r(x)) > 0
        }
        .
    \end{equation*}
    It is easy to check that $\supp(\mu)$ is a closed set
    and that $\mu(K) = 0$ holds for every compact $K \subset X \setminus \supp(\mu)$.
    Further, if $\mu$ is inner regular
    (which follows if $\mu(X)$ is finite, see again \cite[Theorem~2.18]{Rudin1987}),
    we have
    \begin{equation}
        \label{eq:prop_support}
        \mu( X \setminus \supp(\mu) )
        =
        \sup\set{
            \mu(K)
            \given
            K \subset X \setminus \supp(\mu), \; \text{$K$ is compact}
        }
        =
        0 .
    \end{equation}

By $W^{k,r}(\Omega)$, $W_0^{k,r}(\Omega)$, $H^k(\Omega)$ and $H_0^k(\Omega)$ we denote the standard Sobolev spaces.
As usual, $W^{-1,r'}(\Omega)$ and $H^{-1}(\Omega)$ are the dual spaces of $W_0^{1,r}(\Omega)$ and $H_0^1(\Omega)$, respectively,
where $r'$ is the conjugate exponent to $r$.
The $s$-dimensional Hausdorff measure on $\R^d$ is denoted by $\HH^s$ and $\delta_x$ is the Dirac measure at a point $x$.
Moreover, given a function $\varphi : \Omega \to \R$, we define its Lipschitz constant on $\Omega$ by
\begin{equation*}
    \lip_\Omega(\varphi) := \sup\Big\{\tfrac{|\varphi(x) - \varphi(y)|}{\|x-y\|} \colon x, y\in \Omega, \; x\neq y\Big\} \subset [0, \infty] .
\end{equation*}        
Given two linear normed spaces $X$ and $Y$, we denote the space of linear and continuous operators from $X$ to $Y$ by $\LL(X, Y)$. 

Throughout the paper, we assume the following standing assumptions without mentioning them every time:

\begin{assumption}[Standing assumption]\label{assu:standing}
    Let $d_0 \in \N$ and $d_1 \in \{2, 3\}$ be given.
    The set $\Omega \subset \R^{d_1}$ is a bounded Lipschitz domain in the sense of \cite[Chapter~1.2]{grisvard}.
    We assume that $\omega_0 \subset \R^{d_0}$ and $\omega_1 \subset \overline\Omega \subset\R^{d_1}$ are both compact. 
    Moreover, $u^0 \in \frakM(\omega_0)$ with $u^0 \geq 0$ and $c\in C(\omega_0 \times \omega_1)$ are given.
        Finally, we denote by $p_\Omega > d_1$ the exponent from \cite[Theorem~0.5(a)]{JK95}
        and fix an exponent $q$ with
        $q' \in (d_1, p_\Omega)$.
        In case $d_1 = 3$, we additionally assume
        $q' < 6$.

        Moreover,
        the objective $J: W_0^{1,q}(\Omega) \to \R$
        is assumed to be continuous and Gâteaux differentiable
        and the regularization parameter satisfies $\alpha > 0$.
\end{assumption}

Given \cref{assu:standing}, we define the \emph{generalized transportation distance} by
\begin{equation}\label{eq:kant}
\begin{aligned}
    & \wdist^c_{u^0}(u):= \\
    & \qquad \inf\Big\{ \int_{\omega_0\times \omega_1} c \,\d \pi : \pi \in \frakM(\omega_0\times \omega_1),\;
    \pi \geq 0, \; {P_0}_{\#} \pi = u^0, \; {P_1}_{\#} \pi = u \Big\}
\end{aligned}
\end{equation}
along with the usual convention that $\inf \emptyset = \infty$ such that 
$\wdist^c_{u^0}(u) = \infty$, if $u \not\geq 0$ or $|u|(\omega_1) \neq |u^0|(\omega_0)$.
Whenever $u \geq 0$ and $|u|(\omega_1) = |u^0|(\omega_0)$,
it is well known that the minimization problem in \eqref{eq:kant} admits a solution, the so called 
\emph{optimal transport plan}, see, e.g., \cite[Theorem~1.4]{santambrogio}. We denote the set of optimal transport plans associated with 
$c$, $u$, and $u^0$ by 
\begin{equation}
    \KK_c(u^0, u) := 
    \argmin\Big\{ \int_{\omega_0\times \omega_1} c \,\d \pi : 
    \pi \geq 0, \; {P_0}_{\#} \pi = u^0, \; {P_1}_{\#} \pi = u \Big\}.
\end{equation}

For an exponent $r \in (1,\infty)$,
we define the weak Laplacian
$\laplace : W_0^{1,r}(\Omega) \to W^{-1,r}(\Omega)$ by
\begin{equation*}
    - \dual{\laplace y}{v} \coloneqq  \int_\Omega \nabla y \cdot \nabla v \,\d \lambda^{d_1} ,
    \quad \forall y\in W_0^{1,r}(\Omega), \; v \in W_0^{1,r'}(\Omega).
\end{equation*}
The independence of the symbol $\laplace$ on the exponent $r$ will not cause any confusion.
The chosen exponent $p_\Omega$ from \cite[Theorem~0.5(a)]{JK95}
implies that
for every $p \in (p_\Omega', p_\Omega)$
the equation
\begin{equation*}
    -\laplace y = f
\end{equation*}
has a unique solution $y \in W_0^{1,p}(\Omega)$
for every $f \in W^{-1,p}(\Omega)$.
This applies, in particular, to the choice $p = q$.

Then the optimal control problem under consideration reads
\begin{equation}\tag{P}\label{eq:optctrl}
    \left\{\quad
    \begin{aligned}
        \min \quad & J(y) + \alpha\, \wdist^c_{u^0}(u)\\
        \text{w.r.t.} \quad & y \in W_0^{1,q}(\Omega), u \in \frakM(\omega_1), \\
        \text{s.t.} \quad & - \laplace y = E^* u \text{ in } W^{-1,q}(\Omega).
    \end{aligned}
    \right.
\end{equation}

Because of $q' > d_1$,
we have the Sobolev embedding
$W_0^{1,q'}(\Omega) \embed C_0(\Omega)$.
The concatenation with the restriction operator from $C_0(\Omega) \subset C(\overline\Omega)$ to $C(\omega_1)$
yields an operator
$E \colon W_0^{1,q'}(\Omega) \to C(\omega_1)$.
On the right-hand side of the state equation,
the adjoint $E^* \colon \frakM(\omega_1) \to W^{-1,q}(\Omega)$ appears,
i.e., we have
\begin{equation*}
    \dual{E^* u}{v} \coloneqq \int_{\omega_1} v \,\d u, \quad v\in W_0^{1,q'}(\Omega)
\end{equation*}
Note that since $W_0^{1,q'}(\Omega) \embed C_0(\Omega)$,
all mass of $u$ which is located on the boundary $\partial\Omega$
has no effect on the state variable.

\section{Existence of optimal controls}\label{sec:existence}

\begin{lemma}[Existence for the state equation]\label{lem:pdeexist}
    For every $u \in \frakM(\omega_1)$, there exists a unique solution $y \in W_0^{1,q}(\Omega)$ of the state equation given by 
    \begin{equation*}
        y = (-\laplace)^{-1} E^* u.    
    \end{equation*}     
    The solution operator $S := (-\laplace)^{-1} E^* : \frakM(\omega_1)\to W_0^{1,q}(\Omega)$ 
    is linear and completely continuous, i.e., if $u_n \weak^* u$ in $\frakM(\omega_1)$, 
    then $S u_n \to S u $ in $W_0^{1,q}(\Omega)$.    
\end{lemma}

\begin{proof}
        As already mentioned, the choice of $p_\Omega$ and $q$
        implies
        that the state equation admits a unique solution
        in $W^{1,q}_0(\Omega)$
        for every right hand side in
        $W^{-1,q}(\Omega)$.
    Due to the continuous embedding $W_0^{1,q'}(\Omega) \embed C(\overline{\Omega})$
    already mentioned above, $\frakM(\omega_1) \embeds W^{-1,q}(\Omega)^*$ and we obtain 
    existence and uniqueness. 
    The complete continuity of the associated solution operator follows from the 
    compactness of the embedding $W_0^{1,q'}(\Omega) \embed C(\overline{\Omega})$.
\end{proof}

\begin{theorem}[Existence of optimal controls]\label{thm:existoptctrl}
    There exists an optimal control $\bar u \in \frakM(\omega_1)$.
\end{theorem}

\begin{proof}
    The assertion follows in a standard way by the direct method of the calculus of variations. 
    Using the solution operator from \cref{lem:pdeexist}, we can rewrite the optimal control 
    problem in reduced form:
    \begin{equation}
        \eqref{eq:optctrl}
        \quad \Longleftrightarrow \quad
        \min_{u\in \frakM(\omega_1)} \; f(u) := J(S u) + \alpha\, \wdist^c_{u^0}(u).
    \end{equation}
    Let now $\{u_n\}_{n\in \N}$ be an infimal sequence, i.e., 
    \begin{equation}\label{eq:infseq}
        f(u_n) \to j := \inf_{u\in \frakM(\omega_1)} \; f(u) \in \R \cup \{-\infty\}.
    \end{equation}
    Since $\wdist^c_{u^0}(u)  = \infty$, if $|u|(\omega_1) \neq |u^0|(\omega_0)$, it follows that $u_n$ is bounded in 
    $\frakM(\omega_1)$ such that there is a subsequence, denoted by the same symbol to ease notation, 
    such that $u_n \weak^* \bar u$ in $\frakM(\omega_1)$. From \cref{lem:Fconj} in the Appendix, we know that 
    $\wdist^c_{u^0}$ is the Fenchel conjugate of a functional on $C(\omega_1)$ and as such, it is lower semicontinuous w.r.t.\ weak$^*$ convergence, 
    i.e., 
    \begin{equation}\label{eq:wdistlsc}
        \liminf_{n\to\infty} \wdist^c_{u^0}(u_n) \geq \wdist^c_{u^0}(\bar u) .
    \end{equation}
    Concerning the first part of the objective, the absolute continuity of $S$ by \cref{lem:pdeexist} 
    and the assumed continuity of $J$ immediately imply $J(S u_n) \to J(S \bar u)$. 
    Therefore, along the subsequence, $\liminf_{n\to\infty} f(u_n) \geq f(\bar u)$ and, due to \eqref{eq:infseq}, $j$
    is thus finite and $\bar u$ optimal.
\end{proof}

\begin{remark}
    Using the stability of transport plans from \cite[Theorem~5.20]{Vil09}, one can even show that \eqref{eq:wdistlsc} holds with equality and 
    one has the convergence $\wdist^c_{u^0}(u_n) \to \wdist^c_{u^0}(\bar u)$. 
\end{remark}

\begin{corollary}\label{cor:convex}
    If $J$ is convex, then \eqref{eq:optctrl} is a convex minimization problem. 
    Furthermore, if $J$ is strictly convex,
    then the optimal state $\bar y$ and the optimal control restricted to the interior of $\Omega$, i.e., 
    $\bar u \mres (\Omega \cap \omega_1)$ are unique. Thus, if additionally
    $\omega_1 \subset \Omega$, then also the optimal control is unique.
\end{corollary}

\begin{proof}
   The convexity of $\wdist^c_{u^0}$ follows from \cref{lem:Fconj}.
   Consequently, the linearity of $S$ gives the first assertion.
    
    For the second assertion, observe that $S$ is injective, considered as an operator from $\frakM(\Omega \cap \omega_1)$ to $W^{1,q}_0(\Omega)$, which is seen as follows: 
    if the solution $y$ equals zero, then 
    it follows that $\int_{\omega_1} v \d u = 0$ for all $v\in W^{1,q'}_0(\Omega)$
    and the density of $W^{1,q'}_0(\Omega)$ in $C_0(\Omega)$ yields that $u \mres (\Omega\cap \omega_1) = 0$. 
    Therefore, thanks to the linearity of $S$, the strict convexity of $J$ carries over to $J\circ S$ restricted to $\frakM(\Omega\cap \omega_1)$. 
    Along with the convexity of $\wdist^c_{u^0}$, this implies that the optimal state and the optimal control restricted to $\Omega \cap \omega_1$  are unique. 
\end{proof}


\section{First-order optimality conditions}\label{sec:fon}

For the derivation of first-order necessary optimality conditions, we need the concept of $c$-conjugate functions, which is well known in optimal transport theory.

\begin{definition}
    The $c$-conjugate function of a continuous function $\eta \in C(\omega_0)$ is given by 
    \begin{equation*}
        \eta^c : \omega_1 \to \R, \quad \eta^c(\xi) := \inf_{x\in \omega_0} \big( c(x,\xi) - \eta(x)\big).
    \end{equation*}
    Analogously, we define the $\overline{c}$-conjugate function to a given continuous function 
    $\zeta \in C(\omega_1)$ by 
    \begin{equation*}
        \zeta^{\overline c} : \omega_0 \to \R, \quad 
        \zeta^{\overline{c}}(x) := \inf_{\xi\in \omega_1} \big( c(x,\xi) - \zeta(\xi)\big).
    \end{equation*}
    A function $\varphi:\omega_0 \to \overline{\R}$ is called $c$-concave, if a function 
    $\zeta: \omega_1\to \overline{\R}$ exists such that $\varphi = \zeta^{\overline{c}}$. 
    Analogously, a function $\psi: \omega_1 \to \overline{\R}$ is $\overline{c}$-concave, 
    provided that there is $\eta: \omega_0\to \overline{\R}$ such that $\psi = \eta^c$. 
\end{definition}

Note that the conjugate functions only attain finite values due to the compactness of $\omega_0$ and $\omega_1$.
It is known that $\eta^c$ and $\zeta^{\overline c}$ inherit the modulus of continuity of $c$, provided that 
$\eta$ and $\zeta$, respectively, are bounded from above
(see \cite[Section~1.2]{santambrogio}).
Consequently, $\eta^c$ and $\zeta^{\overline c}$ 
are uniformly continuous.

As shown in \cref{sec:subdifftrans}, the Fenchel (pre-)conjugate functional to $\wdist^c_{u^0}$ is given by
\begin{equation}\label{eq:preconj}
    (\wdist^{c}_{u^0})^* : C(\omega_1) \ni \psi \mapsto - \int_{\omega_0} \psi^{\overline c}(x) \,\d u^0(x) \in \R, 
\end{equation}
see \cref{lem:Fconj}.
Note that the integral is well defined and finite, since $\psi^{\overline c}$ is continuous as explained above.

\begin{theorem}[Optimality conditions]\label{thm:fon}
    Let $\bar u$ be a locally optimal control to \eqref{eq:optctrl} with associated state $\bar y = S(\bar u)$. 
    Assume that $J$ is Gâteaux differentiable at $\bar y$.
    Then there exists an adjoint state $p\in W_0^{1,q'}(\Omega) \embed C(\overline{\Omega})$ such that 
    \begin{subequations}\label{eq:optsys}
    \begin{align}
        -\laplace p &= J'(\bar y) \quad \text{in } W^{-1,q'}(\Omega), \label{eq:adjpde}\\
        0 &\in p + \alpha\,\partial \wdist^c_{u^0}(\bar u). \label{eq:pinsubdiff}
    \end{align}
    \end{subequations}
    The latter condition \eqref{eq:pinsubdiff} is fulfilled, 
    if and only if $\big( (-\frac{1}{\alpha} p|_{\omega_1})^{\overline c}, -\frac{1}{\alpha}  p|_{\omega_1} \big)$ 
    solves the dual Kantorovich problem given by 
    \begin{equation}\label{eq:barkantdual}
        \left\{\quad 
            \begin{aligned}
                \max \quad & \int_{\omega_0} \varphi \,\d u^0 + \int_{\omega_1} \psi \,\d \bar u \\
                \textup{w.r.t.} \quad & \varphi \in C(\omega_0), \; \psi \in C(\omega_1), \\
                \textup{s.t.} \quad & \varphi(x) + \psi(\xi) \leq c(x,\xi) \;\forall \, (x,\xi) \in C(\omega_0\times \omega_1).
            \end{aligned}
        \right.
    \end{equation}
    If $J$ is additionally convex, then \eqref{eq:optsys} is also sufficient for optimality.
\end{theorem}

\begin{proof}
    Let $\bar u \in \frakM(\omega_1)$ be locally optimal. Then the convexity of $\wdist^c_{u^0}$ implies for 
    every $u \in \frakM(\omega_1)$ and every $t> 0$ sufficiently small that
    \begin{equation*}
    \begin{aligned}
        0 & \leq \frac{J(S(\bar u + t (u - \bar u))) - J(S\bar u)}{t}
        + \frac{\alpha}{t} \big( \wdist^c_{u^0}(\bar u + t (u - \bar u)) - \wdist^c_{u^0}(\bar u) \big)\\
        &\leq \frac{J(S(\bar u + t (u - \bar u))) - J(S\bar u)}{t}  + \alpha\, \wdist^c_{u^0}(u) - \alpha\,\wdist^c_{u^0}(\bar u)     
    \end{aligned}
    \end{equation*}
    Since $J$ is Gâteaux differentiable at $S\bar u$ by assumption, this leads to 
    \begin{equation*}
        \alpha\, \wdist^c_{u^0}(u) \geq \alpha\, \wdist^c_{u^0}(\bar u) 
        - \dual{J'(S\bar u)}{S( u - \bar u)}_{W^{-1,q'}(\Omega), W_0^{1,q}(\Omega)}
        \quad \forall\, u \in \frakM(\omega_1),
    \end{equation*}
    i.e., $- S^* J'(S\bar u) \in \alpha\, \partial\wdist^c_{u^0}(\bar u)$. Then standard adjoint calculus shows that 
    $p := S^* J'(S\bar u)$
    is the solution of the adjoint equation in \eqref{eq:adjpde}. 
    
    The second statement follows from $p \in W_0^{1,q'}(\Omega) \embed C(\overline\Omega)$
    and the characterization of the 
    subdifferential of $\wdist^c_{u^0}$ from \cref{prop:kantsubdiff} in the appendix.
    
    Finally, if $J$ is convex, \cref{cor:convex} yields that \eqref{eq:optctrl} is a convex problem such that the necessary first-order optimality conditions 
    are also sufficient.
\end{proof}

\begin{remark}\label{rem:pnotcconc}
    It is to be noted that the optimal objective value of the dual Kantorovich problem a
    does not change if one restricts the feasible set to $c$-concave functions. 
    The elements of the subdifferential of $ \wdist^c_{u^0}$ however need not necessarily be $c$-concave 
    and thus the same holds for the adjoint state $p|_{\omega_1}$.
    We will provide an example in \cref{sec:metric}, where $p|_{\omega_1}$ is not $c$-concave, see \cref{ex:linemeasure}.
\end{remark}

Note that the optimality system in \eqref{eq:optsys} is equivalent to the following fixed point equation 
\begin{equation}
    \bar u \in \partial  (\wdist^c_{u^0})^*\big(-\tfrac{1}{\alpha} S^* J'(S\bar u)\big)
\end{equation}
with $ (\wdist^c_{u^0})^*$ as defined in \eqref{eq:preconj}.

\begin{lemma}
    \label{lem:support_plan}
    Let $\bar u$ be a local optimal solution of \eqref{eq:optctrl}
    such that $J$ is Gâteaux differentiable at $\bar y = S(\bar u)$
    and denote by $p$ the adjoint given by \cref{thm:fon}.
    Moreover, we denote by $\bar\pi \in \KK_c(u^0, \bar u)$
    an optimal transport plan.
    Then,
    \begin{equation}
        \label{eq:optimality_support_2}
        \supp(\bar\pi)
        \subset
        \set*{ (x, \xi) \in \omega_0 \times \omega_1
            \given
            c(x, \xi) + \frac1\alpha p(\xi)
            =
            \inf_{\eta \in \omega_1} c(x, \eta) + \frac1\alpha p(\eta)
        }.
    \end{equation}
    Moreover, for every $\xi \in \supp(\bar u)$, there exists $x \in \supp(u^0)$
    such that $(x, \xi) \in \supp(\bar\pi)$.
    This implies
    \begin{equation}
        \label{eq:optimality_support}
        c(x, \xi) + \frac1\alpha p(\xi)
        =
        \inf_{\eta \in \omega_1} c(x, \eta) + \frac1\alpha p(\eta).
    \end{equation}
    i.e., $\xi$ is a minimizer of $c(x, \cdot) + \frac1\alpha p$.
    Similarly,
    for every $x \in \supp(u^0)$, there exists $\xi \in \supp(\bar u)$
    such that
    $(x, \xi) \in \supp(\bar\pi)$, i.e.,
    \eqref{eq:optimality_support} holds.
\end{lemma}
\begin{proof}
    Let $(x, \xi) \in \supp(\bar\pi)$ be given.
    From \cref{thm:fon} we get that
    $(\psi^{\overline{c}}, \psi)$ is a solution of the dual Kantorovich problem,
    where we abbreviate $\psi := - p / \alpha |_{\omega_1}$.
    The celebrated Kantorovich duality yields
    \begin{align*}
        \int_{\omega_0 \times \omega_1} c(x,\xi) \, \d\bar\pi(x,\xi)
        &=
        \int_{\omega_0} \psi^{\overline{c}}(x) \,\d u^0(x)
        +
        \int_{\omega_1} \psi(\xi) \,\d \bar u(\xi)
        .
    \end{align*}
    Since the marginals of $\bar \pi$ are $u^0$ and $\bar u$, this yields
    \begin{equation*}
        \int_{\omega_0 \times \omega_1} \psi^{\overline{c}}(x) + \psi(\xi) - c(x,\xi) \, \d\bar\pi(x,\xi)
        =
        0.
    \end{equation*}
    The definition of the $\overline{c}$-conjugate
    implies that the integrand is non-positive.
    Since $(x,\xi) \in \supp(\bar\pi)$, the continuity of the involved functions
    implies
    \begin{equation*}
        \psi^{\overline{c}}(x) + \psi(\xi) - c(x,\xi)
        =
        0.
    \end{equation*}
    Using again the definition of the $\overline{c}$-conjugate
    shows \eqref{eq:optimality_support_2}.

    Now, let $\xi \in \supp(\bar u)$ be given
    and denote by $\bar\pi \in \KK_c(u^0, \bar u)$
    an optimal transport plan.
    From \cref{lem:pisupp} we get the existence
    of $x \in \supp(u^0)$
    such that $(x, \xi) \in \supp(\bar\pi)$.
    The first part of the proof yields \eqref{eq:optimality_support}.
    The last assertion follows \emph{mutatis mutandis}.
\end{proof}

In the following sections, we aim to deduce structural information on the optimal solution of \eqref{eq:optctrl}
that is inherent in the optimality system \eqref{eq:optsys}.  We distinguish between different settings depending on the 
transportation costs, the regularity of the prior $u^0$ and the objective $J$.


\section{Tracking type objective and boundedness of the optimal state}\label{sec:tracking}

We begin our investigations on the consequences of the optimality conditions in \cref{thm:fon} with the popular case 
of a tracking type objective of the form
\begin{equation}\label{eq:trackingOmega}
    J(y) \coloneqq \frac{1}{2} \, \|y - y_d\|_{L^2(\Omega)}^2
\end{equation}
with a given desired state $y_d \in L^2(\Omega)$. 
Note that, according to \cref{lem:pdeexist}, the state satisfies $y\in W^{1,q}_0(\Omega)$ and, by Sobolev embeddings, 
this ensures $y\in L^s(D)$ with $s =  d_1 q/(d_1 -q)$ 
and we have $s \ge 2$ due to the assumptions on $q$ in \cref{assu:standing}.
Consequently, $J \colon W_0^{1,q}(\Omega) \to \R$ is well defined, continuous and Gâteaux differentiable.
Note moreover that \eqref{eq:optctrl} becomes a convex problem in this case, 
cf.\ \cref{cor:convex}, so that we do not have to distinguish between local and global minimizers.
	In this section, we make use of the following assumption.

\begin{assumption}\label{assu:contcosts}
    We suppose that the transportation costs and the desired state satisfy $c\in C^2(\omega_0\times \omega_1)$ 
    and $y_d \in C(\omega_1)$.
    The functional $J$ is given by \eqref{eq:trackingOmega}.
\end{assumption}

In all what follows, let us consider an arbitrary optimal control $\bar u \in \frakM(\omega_1)$ with associated state $\bar y = S(\bar u)$. 
Note that, by \cref{cor:convex}, there is a unique optimal state $\bar y$, whereas the optimal control is only unique in $\Omega\cap \omega_1$.
We will show that we can use the Green potential associated with $\bar u$
as a representative of (the equivalence class) $\bar y$. 
A similar technique is also employed in \cite{PieperVexler2013}, but only concerning measures with compact support. For that reason, 
we present the underlying analysis in detail.
We start with the definition of the Green function, which is, as usual, defined by
\begin{equation}
    G_\Omega \colon \Omega \times \Omega \to [0,\infty], \quad G_\Omega(x, \xi) \coloneqq \Phi_x(\xi) - h_x(\xi), \quad x, \xi \in \Omega,
\end{equation}
where $\Phi_x$ is the fundamental solution with pole $x$ and $h_x$ is the greatest harmonic minorant of $\Phi_x$ on $\Omega$, 
which exists due to \cite[Theorem~3.6.3]{AG01}.
From
\cite[Theorem~4.1.9]{AG01},
we get that $G_\Omega$ is superharmonic on $\Omega \times \Omega$.
Consequently, it is lower semicontinuous and Borel measurable from $\Omega \times \Omega$
to $[0,\infty]$.

    The next two lemmas are of independent interest.
    Therefore, we state them independent of our standing assumption.
    It seems that the assertions of these lemmas are well known,
    but we were not able to find a suitable reference.
\begin{lemma}\label{lem:poissongreen}
    Let $\Omega \subset \R^d$, $d = 2,3$, be a bounded Lipschitz domain
    and fix $q \in (1, d/(d-1))$.
    Then, 
    for every $x\in \Omega$, the Green function $\Omega \ni \xi \mapsto G_\Omega(x, \xi) \in [0,\infty]$ is an element of $W^{1,q}_0(\Omega)$ and 
    solves the Poisson equation associated with $\delta_{x}$ in the weak sense, i.e., 
    \begin{equation}\label{eq:poissongreen}
        - \laplace G_\Omega(x, \cdot ) = E_\Omega^* \delta_{x} \quad \text{in } W^{-1,q}(\Omega),
    \end{equation}
    where $E_\Omega \colon W_0^{1,q'}(\Omega) \to C(\overline{\Omega})$
    is the Sobolev embedding.
\end{lemma}

\begin{proof}
    W.l.o.g.,
    we can assume that $q > 2 d /(d+2)$
    and $q > p_\Omega'$, where $p_\Omega > d$ is the exponent from
    \cite[Theorem~0.5(a)]{JK95}.
    Let us denote the solution of \eqref{eq:poissongreen} by $w \in W^{1,q}_0(\Omega)$, i.e., 
    \begin{equation*}
        - \laplace w= \delta_{x} \quad \text{in } W^{-1,q}(\Omega)
    \end{equation*}
    and define $h \coloneqq \Phi_x - w$. Note that this solution exists by \cref{lem:pdeexist}.
    Straightforward computation shows that $\Phi_x \in W^{1,q}(\Omega)$ and thus 
    $h \in W^{1,q}(\Omega)$. Moreover, since $\Phi_x$ satisfies Poisson's equation with right hand side $\delta_x$ in the distributional sense, we find
    \begin{equation*}
        \int_\Omega h \,\laplace \varphi\, \d \lambda^d = 0 \quad \forall\, \varphi\in C^\infty_c(\Omega)
        .
    \end{equation*}
    Weyl's lemma \cite[Proposition~2.14]{Ponce} yields that
    $h$ admits a harmonic representative (which we also denote by $h$),
    in particular, this yields $h \in C^\infty(\Omega)$. 
    Moreover, $w$ is non-negative by \cref{lem:comparison} and thus, 
    $h \leq \Phi_x$ such that $h$ is a harmonic minorant of $\Phi_x$.

    To conclude the claim, we prove that $h$ is the greatest harmonic minorant of $\Phi_x$ on $\Omega$. 
    For this purpose, first note that a classical localization argument, cf.\ \cref{lem:localreg} in the Appendix, implies that $w$ is continuous close to the boundary, 
    since $\dist(x, \partial\Omega) > 0$. For the same reason $\Phi_x$ is continuous close to $\partial\Omega$, too, and so is $h$.
    Now let $h_x$ be the greatest harmonic minorant of $\Phi_x$ and pick an arbitrary point $\xi \in \partial\Omega$ on the boundary and an arbitrary sequence 
    $\{\xi_n\} \subset \Omega$ converging to that point. Then the continuity of $\Phi_x$, $w$, and $h$ close to the boundary yields 
    \begin{equation*}
    \begin{aligned}
        h(\xi)
        & = \lim_{n\to \infty} h(\xi_n)
        \leq \liminf_{n\to\infty} h_x(\xi_n) \leq \limsup_{n\to\infty} h_x(\xi_n)
        \\&
        \leq \lim_{n\to\infty}  \Phi_x(\xi_n) =  \Phi_x(\xi)
        = \Phi_x(\xi) - w(\xi)
        = h(\xi) 
        ,
    \end{aligned}
    \end{equation*}
    where we used that $w$ satisfies zero Dirichlet boundary conditions, that $h$ is a harmonic minorant, 
    and that $h_x$ is the greatest harmonic minorant of $\Phi_x$.
    Therefore, $h_x$ is continuous at the boundary, too, and satisfies $h_x = h$ on $\partial\Omega$. The maximum principle for harmonic functions that 
    are continuous on $\overline{\Omega}$ thus implies $h_x = h$ in $\Omega$.
    Consequently, the function $G_\Omega(x, \cdot)$ is a representative of the equivalence class $w$.
\end{proof}

Given the Green function, the Green potential associated with $\bar u$ is
the function $\bar\yp \colon \Omega \to [0, \infty]$
defined by 
\begin{equation}\label{eq:greenpotential}
    \bar\yp(x) = \int_{\Omega \cap \omega_1} G_\Omega(x, \xi)\,\d \bar u(\xi).
\end{equation}
Note that the integral is well defined and $\bar\yp$ is Borel measurable
owing to Fubini's theorem and
to the Borel measurability of $G_\Omega$. The next lemma shows that $\bar\yp$ is just a representative of the equivalence class 
$\bar y$. 

\begin{lemma}\label{lem:sobolevgreen}
    Let $\Omega \subset \R^d$, $d = 2,3$, be a bounded Lipschitz domain.
        We denote by $p_\Omega > d$
        the exponent from \cite[Theorem~0.5(a)]{JK95}
        and fix $q \in (p_\Omega' , d/(d-1))$.
    Further, let $\mu \in \frakM(\overline\Omega)$ be a non-negative Borel measure
    and let $w \in W^{1,q}_0(\Omega)$ be the solution of
    \begin{equation}
        \label{E5.5}
        - \laplace w = E_\Omega^* \mu \quad \text{in } W^{-1,q}(\Omega).
    \end{equation}
    Moreover, denote the (pointwise defined) Green potential of $\mu$ by $\wp$, i.e.,
    \begin{equation}
        \label{E5.6}
        \wp(x) \coloneqq \int_\Omega G_\Omega(x, \xi) \,\d\mu(\xi) \in [0,\infty]
        \quad \forall\, x\in \Omega.
    \end{equation}
    Then the function $\wp$ is a representative of the equivalence class $w$.
\end{lemma}
\begin{proof}
    From the Sobolev embedding and $q' > d$,
    we get that $W_0^{1,q'}(\Omega)$ embeds into the Hölder space
    $C_0^{0,\alpha}(\Omega)$ for some $\alpha > 0$.
    Consequently,
    the map $\overline\Omega \ni x \mapsto E_\Omega^* \delta_x \in W^{-1,q}(\Omega)$ is Hölder continuous.
    Together with the solution map $(-\laplace)^{-1} \colon W^{-1,q}(\Omega) \to W_0^{1,q}(\Omega)$
    from of Poisson's equation, we obtain a function
    \begin{equation*}
        g_\Omega \in C( \overline\Omega , W_0^{1,q}(\Omega) ),
        \qquad
        g_\Omega(x) := (-\laplace)^{-1} E_\Omega^* \delta_x
        .
    \end{equation*}
    From \cref{lem:poissongreen}, we get that
    $G_\Omega(x,\cdot)$ is a representative of $g_\Omega(x)$ for every $x \in \Omega$.
    Moreover, for $x \in \partial\Omega$, we have $g_\Omega(x) = 0$, since $E_\Omega^* \delta_x = 0$ in this case.

    Since $g_\Omega$ is continuous, it is Bochner measurable from $\overline\Omega$
    into $W_0^{1,q}(\Omega)$.
    In particular, we can define
    \begin{equation*}
        w := \int_{\Omega} g_\Omega(x) \d\mu(x) \in W_0^{1,q}(\Omega).
    \end{equation*}
    Since the Bochner integral commutes with bounded linear operators, we obtain
    \begin{equation*}
        -\laplace w
        =
        \int_{\Omega} (-\laplace g_\Omega)(x) \d\mu(x)
        =
        \int_{\Omega} E_\Omega^* \delta_x \d\mu(x)
        =
        E_\Omega^* \mu
        ,
    \end{equation*}
    where the last equation follows from testing with functions in $W_0^{1,q'}(\Omega)$.
    This shows that $w$ is the solution of \eqref{E5.5}.
    Further, \cref{lem:pdeexist} implies uniqueness,
    i.e., $w$ is the unique solution.

    In order to check that $\wp$ is a representative of $w$,
    we take an arbitrary Borel set $B \subset \Omega$.
    Then, since $G_\Omega(x,\cdot)$ is a representative of $g_\Omega(x)$, it holds that
    \begin{align*}
        \int_B w(\xi) \,\d\lambda^{d}(\xi)
        &=
        \int_{\Omega} \int_B [g_\Omega(x)](\xi) \, \d\lambda^{d}(\xi) \, \d\mu(x)
        \\&
        =
        \int_{\Omega} \int_B G_\Omega(x,\xi) \, \d\lambda^{d}(\xi) \, \d\mu(x)
        \\&
        =
        \int_B \int_{\Omega} G_\Omega(x,\xi) \, \d\mu(x) \, \d\lambda^{d}(\xi)
        =
        \int_B \wp(\xi) \, \d\lambda^{d}(\xi)
        .
    \end{align*}
    Note that the first equality again uses that the Bochner integral commutes with bounded linear operators
    and the second-to-last equality uses again Fubini's theorem.
    Since the Borel set $B$ was arbitrary, this shows the claim.
\end{proof}

\begin{remark}
    \label{rem:higher_dimensions}
    It is possible to generalize \cref{lem:sobolevgreen} to higher dimensions.
    This needs two ingredients.
    First, one has to use a substitute for
    \cite[Theorem~0.5(a)]{JK95}
    which implies that $-\laplace \colon W_0^{1,q}(\Omega) \to W^{-1,q}(\Omega)$
    is invertible for some $q < d / (d-1)$.
    Moreover, one has to generalize \cref{lem:localreg}
    to higher dimensions.
    Both of these results might need some higher regularity of the boundary $\partial\Omega$.
\end{remark}

\cref{lem:sobolevgreen} implies that the Green potential $\bar\yp$ from \eqref{eq:greenpotential} is indeed a representative of the optimal state $\bar y$. 
By \cite[Theorem~3.3.1]{AG01}, $\bar\yp$ is superharmonic and thus lower semicontinuous on $\Omega$.
This can be used to prove the following:

\begin{lemma}\label{lem:yboundsupp}
    Under \cref{assu:contcosts} there holds that 
    \begin{equation}
        \label{Gleichung_5_6}
        \bar\yp(\xi) \leq y_d(\xi) + \alpha\, \max_{x\in \omega_0} |\laplace c(x,\xi)| \quad \forall\, \xi \in  \supp(\bar u) \cap \interior(\omega_1)
        ,
    \end{equation}
    where the Laplacian acts on the second argument of $c$.
\end{lemma}

\begin{proof}
    We argue by contradiction and assume that there is a point $\xi \in \supp(\bar u) \cap \interior(\omega_1)$ such that 
    \begin{equation}\label{eq:yboundcontra}
        - \laplace c(x_\xi, \xi) + \frac{1}{\alpha} \big(\bar\yp(\xi) - y_d(\xi)\big) > 0,
    \end{equation}
    where $x_\xi \in \supp(u^0)$ is the point given by \cref{lem:support_plan}.
    Then, by lower semicontinuity of $\bar\yp$, \cref{assu:contcosts}, and $\xi \in \interior(\omega_1)$, there exists $r > 0$ such that
    \begin{equation}\label{eq:laplacewonball}
        - \laplace c(x_\xi, \eta) + \frac{1}{\alpha} \big(\bar\yp(\eta) - y_d(\eta)\big) > 0 \quad \forall \, \eta \in B_r(\xi) \subset \omega_1.
    \end{equation}
    Moreover, \eqref{eq:optimality_support} implies that
    the function $z \in H^1(B_r(\xi)) \cap C(\overline{B_r(\xi)})$, defined by 
    \begin{equation*}
        z: B_r(\xi) \ni \eta \mapsto c(x_\xi, \eta ) + \frac1\alpha p(\eta)
    \end{equation*}
    attains its minimum at $\xi$. Furthermore, the adjoint equation \eqref{eq:adjpde} along with \eqref{eq:trackingOmega} implies
    \begin{equation*}
        - \laplace z = - \laplace c(x_\xi, \cdot ) - \frac{1}{\alpha} \, \laplace p  =  - \laplace c(x_\xi, \cdot ) + \frac{1}{\alpha} (\bar\yp - y_d) 
        \quad \text{in } H^{-1}(B_r(\xi)). 
    \end{equation*}
    Hence, in view of \eqref{eq:laplacewonball}, we obtain (in the distributional sense) 
    \begin{equation}\label{eq:laplacewonball2}
        - \laplace z > 0 \quad \text{in } B_r(\xi)
    \end{equation}
    and consequently, by the strong maximum principle, cf., e.g., \cite[Theorem~8.19]{GT01}, and the continuity of $z$, there holds
    \begin{equation*}
        \inf_{\eta \in B_r(\xi)} z(\eta) = \inf_{\eta \in \partial B_r(\xi)} z(\eta).
    \end{equation*}
    Due to \eqref{eq:laplacewonball2}, $z$ cannot be constant on $B_r(\xi)$, and thus, the infimum of $z$ is not attained in the open ball $B_r(\xi)$ contradicting \eqref{eq:optimality_support}.
    Hence, \eqref{eq:yboundcontra} is wrong, which in turn implies the assertion.
\end{proof}

\begin{corollary}\label{cor:nodiracinterior}
    Let \cref{assu:contcosts} be fulfilled. 
    Then, every optimal control $\bar u$ is non-atomic in $\interior(\omega_1)$. 
\end{corollary}

\begin{proof}
    Suppose by contrary that $\bar u$ has an atom $a \in \BB(\omega_1)$ in $\interior(\omega_1)$. 
    By \cite[Section~2.1.6, Theorem~2]{kadets}, every atom of $\bar u$ is a singleton, so that $a$ can be identified with a single point with $\bar u(a) > 0$. 
    Then, $\bar u - \bar u(a) \delta_a \geq 0$ along with the non-negativity of the Green function gives $\bar\yp(x) - \wp(x) \geq 0$ 
    for all $x\in \Omega$, where, as before, $\bar\yp$ and $\wp$ are the Green potentials of $\bar u$ and $\bar u(a) \delta_a$, respectively. 
    This however implies
    \begin{equation*}
        \bar\yp(a)  \geq \wp(a) = \bar u(a) \int_\Omega G_\Omega(a, \xi) \,\d\delta_{a}(\xi) = \bar u(a) \,G_\Omega(a,a) =  \infty
    \end{equation*}
    in contradiction to \cref{lem:yboundsupp}.
\end{proof}

\begin{remark}
    The result of \cref{cor:nodiracinterior} is remarkable: in the setting of \cref{assu:contcosts}, 
    even if $\omega_1 = \omega_0$ and the prior $u^0$ just consists of 
    Dirac measures in the interior of $\omega_1$, an optimal control can only have Diracs on the boundary of $\omega_1$,
    no matter how large the Tikhonov parameter is chosen.
    Moreover, if $\bar u$ contains a Dirac measure at $x \in \partial\omega_1$,
    it induces a singularity of $\bar\yp$ at $x$.
    Consequently, the inequality in \eqref{Gleichung_5_6} is violated for all $\xi$
    in a neighborhood $N(x)$ of $x$ and, consequently,
    $N(x) \cap \interior(\omega_1)$ cannot intersect $\supp(\bar u)$.
\end{remark}

In order to show that the bound on the optimal state $\bar y$ on $\supp(\bar u) \cap \interior(\omega_1)$ from \cref{lem:yboundsupp} 
transfers to the entire domain, we need the following result, see \cite[Theorem~5.4.8]{Helms}.

\begin{theorem}[Maria--Frostman domination principle]\label{lem:mariaforstmann}
    Let $\mu \in \frakM(\overline\Omega)$ be a non-negative measure
    and denote the associated Green potential again by $\wp$,
    see \eqref{E5.6}.
    If there exists a constant $M > 0$ such that $\wp(x) \leq M$ for all $x\in \supp(\mu) \cap \Omega$,
    then $\wp(x) \leq M$ for all $x\in \Omega$.
\end{theorem}

\begin{theorem}\label{thm:statebound}
    In addition to \cref{assu:contcosts}, we assume
    $\omega_1 = \overline{\Omega}$.
    Then the optimal state 
    $\bar y$ is a function in $L^\infty(\Omega) \cap H_0^1(\Omega)$ and satisfies
    \begin{align*}
        \norm{\bar y}_{L^\infty(\Omega)}
        &\leq
        \norm{y_d}_{C(\overline{\Omega})} + \alpha\, \norm{\laplace c}_{C(\omega_0 \times \overline{\Omega})} 
        ,
        \\
        \norm{\bar y}_{H_0^1(\Omega)}^2
        &\le
        \norm{\bar y}_{L^\infty(\Omega)}
        \norm{\bar u}_{\frakM(\Omega)}
        ,
    \end{align*}
    where the Laplacian acts on the second argument of $c$.
    Consequently, every optimal control $\bar u$ induces an element of $H^{-1}(\Omega)$
    and does not charge subsets of $\Omega$ of $H_0^1(\Omega)$-capacity zero.
\end{theorem}

\begin{proof}
    From \cref{lem:yboundsupp}, it follows that $\bar\yp(x) \leq \|y_d\|_{C(\overline{\Omega})} + \alpha\, \|\laplace c\|_{C(\omega_0 \times \overline{\Omega})}$ 
    for all $x \in \supp(\bar u) \cap \Omega$.
    The $L^\infty(\Omega)$-bound then follows from \cref{lem:mariaforstmann} along with the non-negativity of $\bar\yp$.

    In order to check $\bar y \in H_0^1(\Omega)$ and the associated norm bound,
    we use \cite[Lemma~5.8]{Ponce}.
    In this lemma, it is assumed that the boundary of $\Omega$ is smooth,
    but this only used to check the unique solvability of the Poisson equation in the proof.
    Using our \cref{lem:pdeexist}, this is also possible on Lipschitz domains.

    The $H^1_0$-regularity of $\bar y$ directly implies $E_\Omega^* \bar u \in H^{-1}(\Omega)$.
    The last claim then follows from \cite[Lemma~6.55]{BonnansShapiro2000}.
\end{proof}

In \cref{cor:nodiracinterior} we have seen that, under \cref{assu:contcosts}, an optimal control $\bar u$ has no Diracs in the interior of $\omega_1$. 
In the prominent example of metric transportation costs however, where $c(x, \xi) = \|x - \xi\|$, \cref{assu:contcosts} is not satisfied. 
The next theorem shows that, in this case, a Dirac in $\bar u$ can only appear if the prior has an atom at the same point.

\begin{theorem}\label{thm:nodiractransport}
    Assume that $d_0 = d_1 =: d$. 
    Let $J$ be of tracking type, i.e., given by \eqref{eq:trackingOmega}, with a desired state $y_d$ that satisfies
    \begin{equation}\label{eq:ydreg}
        y_d \in 
        \begin{cases}
            L^s(\Omega), \; s > 3, & \text{if } d = 3,\\
            C^{0,\tau}(\Omega), \; \tau > 0, & \text{if } d = 2.
        \end{cases}
    \end{equation}
    Suppose moreover that, for every $x\in \omega_0$, the mapping $\omega_1 \ni \eta \mapsto c(x, \eta) \in \R$ 
    is twice continuously differentiable at every $\eta \in \interior(\omega_1) \setminus\{x\}$. 
    Further, let $\xi \in \interior(\omega_1)$ be an atom of an optimal control $\bar u$.
    Then, $(x, \xi) \in \supp(\bar\pi)$ implies $x = \xi$, that is, no mass from any other point is transported to $\xi$.
    Thus, $\bar u$ can only contain a Dirac at $\xi \in \interior(\omega_1)$, if $\{\xi\}$ is an atom of $u^0$. 
    Consequently, if $\omega_0 \cap \interior(\omega_1) = \emptyset$, then every optimal control can only have 
    atoms on the boundary of $\omega_1$.
\end{theorem}
\begin{proof}
    Suppose that $\bar u$ contains an Dirac at $\xi \in \interior(\omega_1)$ with weight $\beta > 0$, i.e., 
    \begin{equation*}
        \bar u = \beta \, \delta_{\xi} + \bar u \mres (\omega_1 \setminus \{\xi\}).
    \end{equation*}
    According to \cref{lem:poissongreen}, the optimal state $\bar y$ is then given by 
    \begin{equation*}
        \bar y = \beta \, ( \Phi_\xi + h_\xi ) + \underbrace{(-\laplace)^{-1} E^* \big(\bar u \mres (\omega_1 \setminus \{\xi\})\big)}_{\displaystyle{=: \hat y}},
    \end{equation*}
    where  $\Phi_\xi$ is the fundamental solution with pole $\xi$ and $h_\xi$ its greatest harmonic minorant on $\Omega$.
    By \cref{lem:comparison} in the Appendix, we know that $\hat y \in W^{1,q}_0(\Omega)$ is non-negative. 
    Moreover, by \cref{thm:fon}, the adjoint state $p$ is given by 
    \begin{equation}\label{eq:defp123}
        p = \beta (-\laplace)^{-1}(\Phi_\xi)  + (-\laplace)^{-1} (\beta \, h_\xi - y_d) + (-\laplace)^{-1} \hat y =: p_1 + p_2 + p_3,
    \end{equation}
    where, with a little abuse of notation, $(-\laplace)^{-1}$
    can be considered as an operator from $W^{-1,q'}(\Omega)$
    to $W^{1,q'}_0(\Omega)$,
    due to the embedding
    $W_0^{1,q}(\Omega) \embeds L^2(\Omega) \embeds W^{-1,q'}(\Omega)$.
    Note that $h_\xi$ is continuous up to the boundary of $\Omega$, since $\xi \in \interior(\omega_1) \subset  \Omega$, 
    cf.\ the proof of \cref{lem:poissongreen}, so that $p_2$ is well defined.
    Let us consider the three different contributions separately. For $p_1$ we find
    \begin{equation*}
        p_1 = \beta \big( \Psi_\xi |_{\Omega} + h \big),
    \end{equation*}
    where $\Psi_\xi$ is a distributional solution of
    \begin{equation*}
        -\laplace \Psi_\xi = \Phi_\xi \quad\text{on } \R^d
    \end{equation*}
    and $h$ is the harmonic extension of $-\Psi_\xi|_{\partial\Omega}$ to $\Omega$. While $h$ is arbitrarily smooth in a neighborhood of $\xi$, 
    a solution $\Psi_\xi$ can be computed exactly. By the radial symmetry of the fundamental solution, 
    every such function is of the form $\Psi_\xi(\eta) = \psi_\xi(\norm{\eta - \xi})$ with a radially symmetric function $\psi_\xi$ that solves 
    \begin{equation}\label{eq:laplacepolar3d}
        -\frac{1}{r^2} \frac{\partial}{\partial r}
        \parens*{ r^2 \, \frac{\partial\psi_\xi}{\partial r}}
        =
        \frac{1}{4 \pi} r^{-1}, \quad r > 0 
    \end{equation}
    in the three-dimensional case and 
    \begin{equation}\label{eq:laplacepolar2d}
        -\frac{1}{r} \frac{\partial}{\partial r}
        \parens*{ r \, \frac{\partial\psi_\xi}{\partial r}}
        =
        - \frac{1}{2 \pi}\, \ln(r), \quad r > 0 ,
    \end{equation}
    in case of $d = 2$. 
    Let us first focus on \eqref{eq:laplacepolar3d}. Straightforward computations shows that any solution to this equation is given by
    \begin{equation*}
        \psi_\xi(r) = -\frac{1}{8 \pi} \, r + C_1 \, r^{-1} + C_2, \quad C_1, C_2 \in \R
    \end{equation*}
    Due to $p_1 \in W^{1,q'}_0(\Omega) \embed C(\overline{\Omega})$, the constant $C_1$ must necessarily vanish, while $C_2$ can be compensated by $h$ so that 
    \begin{equation}\label{eq:p13d}
        p_1(\eta) = \beta\Big(- \frac{1}{8 \pi} \, \|\eta - \xi\| + h(\eta)\Big), \quad \text{if } d = 3 .
    \end{equation}
    In the two-dimensional case, where $\psi_\xi$ is determined by \eqref{eq:laplacepolar2d}, we obtain     
    \begin{equation*}
        \psi_\xi(r) = \frac{1}{8\pi}\,r^2 \big(\ln(r) - 1 \big) + C_1 \, \ln(r) + C_2,\quad C_1, C_2 \in \R
    \end{equation*}
    (with the convention $r^2 \ln(r) := 0$, if $r = 0$).
    Again, due to $p_1 \in C(\overline{\Omega})$, we again have $C_1 = 0$, while $C_2$ is again compensated by the harmonic part $h$.  
    Therefore, in the two-dimensional case, $p_1$ is given by
    \begin{equation}\label{eq:p12d}
        p_1(\eta) = \beta \Big(\frac{1}{8 \pi} \, \|\eta - \xi\|^2 \big(\ln(\|\eta - \xi\|) - 1 \big)  + h(\eta)\Big), \quad \text{if } d = 2 .
    \end{equation}

    Concerning $p_2$ in case $d = 3$,
    the regularity of $y_d$ and standard interior regularity results, cf., e.g., \cite[Theorem~9.11]{GT01},  imply
    \begin{equation}\label{eq:p2est}
    \begin{aligned}
        \| p_2 \|_{C^{1, \gamma}(B_\rho(\xi))} 
        & \leq C\, \|p_2\|_{W^{2,s}(B_\rho(\xi))} \\
        & \leq C\, \big( \beta \|h_\xi\|_{C(\overline{\Omega})} +  \|y_d\|_{L^s(\Omega)} \big)
        \quad\text{with}\quad \gamma = 1 - \frac{d}{s},    
    \end{aligned}
    \end{equation}
    where $\rho> 0$ is such that $\overline{B_{2\rho}(\xi)} \subset \interior(\omega_1) \subset  \Omega$. 
    Note that our assumption on $s$ implies $\gamma > 0$. 
    In the two dimensional case, we argue as follows: We already know that $p_2$ is continuous on the whole $\overline{\Omega}$. 
    Moreover, $p_2$ solves
    \begin{equation*}
        - \laplace v = \beta h_\xi - y_d \quad \text{in } B_{2\rho}(\xi), \quad
        v = p_2 \quad \text{on } \partial B_{2\rho}(\xi) .
    \end{equation*}
    As the right hand side is H\"older continuous by assumption, cf.\ \eqref{eq:ydreg}, and since $h_\xi$ is harmonic,
    and the boundary data are continuous, \cite[Theorem~4.3]{GT01} implies that this equation admits a unique classical solution 
    $v \in C^2(B_{2\rho}(\xi))$. Therefore, $p_2 \in C^2(B_{2\rho}(\xi))$ and hence, the assumptions of \cite[Theorem~4.6]{GT01} are fulfilled, which implies
    \begin{equation*}
    \begin{aligned}
        \| p_2 \|_{C^{1, 1}(B_\rho(\xi))} 
        & \leq C\, \|p_2\|_{C^{2,\tau}(B_\rho(\xi))} \\
        & \leq C\big( \|p_2\|_{C(\overline{B_{2\rho}(\xi)})} + \|\beta h_\xi - y_d\|_{C^{0,\tau}(B_{2\rho}(\xi))} \big) \\
        & \leq C \big(\|h_\xi\|_{C^{0, \tau}(B_{2\rho}(\xi))} + \|h_\xi\|_{C(\overline{\Omega})} +  \|y_d\|_{C^{0,\tau}(\Omega)} \big)  .
    \end{aligned}
    \end{equation*}
        
    Moreover, due to the non-negativity of $\hat y$, the third contribution  
    \begin{equation*}
        p_3(\eta) = \int_\Omega G_\Omega(\eta, \zeta)\, \hat y(\zeta) \,\d\lambda^3(\zeta)
    \end{equation*}        
    is superharmonic, see, e.g., \cite[Theorem~3.3.1]{AG01}.
    Therefore, for every $\rho>0$ such that $\overline{B_\rho(\xi)} \subset  \interior(\omega_1) \subset \Omega$ and every 
    $\varphi \in C^{1,\kappa}(B_\rho(\xi))$, $\kappa \in (0,1]$, and $r \in (0,\rho)$, there holds that
    \begin{equation*}
    \begin{aligned}
        & \frac{1}{\HH^2(\partial B_r(\xi))} \int_{\partial B_r(\xi)} \big(p_3(\eta) + \varphi(\eta)\big) \d\HH^2(\eta) \\
        & \; \leq p_3(\xi) + \varphi(\xi) + \frac{1}{\HH^2(\partial B_r(\xi))} \, \nabla \varphi(\xi) \cdot \underbrace{\int_{\partial B_r(\xi)} (\eta - \xi)\,\d \HH^2(\eta)}_{=0}
        +  \|\varphi\|_{C^{1,\kappa}(B_\rho(\xi))} \, r^{1+\kappa} .
    \end{aligned}
    \end{equation*}
    Hence, for all $r \in (0,\rho)$, we obtain
    \begin{equation*}
         \int_{\partial B_r(\xi)} \big(p_3(\eta) + \varphi(\eta) \big) \d\HH^2(\eta) \leq \HH^2(\partial B_r(\xi)) 
         \Big(p_3(\xi) + \varphi(\xi) +  \|\varphi\|_{C^{1,\kappa}(B_\rho(\xi))} \, r^{1+\kappa} \Big)
    \end{equation*}
    and therefore, there exists a point $\eta_r \in \partial B_r(\xi)  \subset \interior(\omega_1)$ such that 
    \begin{equation}\label{eq:superharmest}
        p_3(\eta_r) + \varphi(\eta_r) \leq p_3(\xi) + \varphi(\xi) +   \|\varphi\|_{C^{1,\kappa}(B_\rho(\xi))} \, r^{1+\kappa} .
    \end{equation}
    
    Now,
    due to $\xi \in \supp(\bar u)$,
    \cref{lem:support_plan}
    implies the existence of $x \in \supp(u^0)$ such that
    \begin{equation}\label{eq:ximin}
        \xi \in \argmin_{\eta \in  \omega_1}
        \Big\{
            \underbrace{c(x, \eta) + \frac{1}{\alpha}\, \big(p_1(\eta) + p_2(\eta) + p_3(\eta)\big)}_{\displaystyle{=: f_x(\eta)}} \Big\} .
    \end{equation}
    Let us assume that $x\neq \xi$ and investigate the objective $f_x$. By assumption, $c(x, \,\cdot\,) \in C^2(B_{2\rho}(\xi))$, provided that $\rho > 0$ is sufficiently small. 
    Thus we obtain 
    \begin{equation*}
        \varphi := \alpha \, c(x, \,\cdot\,) + \beta\, h + p_2 \in C^{1,\kappa}(B_\rho(\xi))
        \quad \text{with} \quad \kappa = 
        \begin{cases}
            1, & \text{if } d = 2,\\
            \gamma, & \text{if } d = 3.
        \end{cases}
    \end{equation*}        
    Therefore, in the three-dimensional case, we deduce from \eqref{eq:superharmest} and \eqref{eq:p13d} that, for all $r < \rho$, it holds
    \begin{equation}\label{eq:fxabstieg}
    \begin{aligned}
        f_x(\eta_r) & = - \frac{\beta}{8\pi\alpha} \, \|\eta_r - \xi\|  + \frac{1}{\alpha} \big( p_3(\eta_r) + \varphi(\eta_r)\big) \\
        & \leq \frac{1}{\alpha}\big(p_3(\xi) + \varphi(\xi) + \|\varphi\|_{C^{1,\gamma}(B_\rho(\xi))} \, r^{1+\gamma}\big) - \frac{\beta}{8\pi\alpha} \, r \\
        &= f_x(\xi) + \frac{1}{\alpha}  \, \|\varphi\|_{C^{1,\gamma}(B_\rho(\xi))} \, r^{1+\gamma} - \frac{\beta}{8\pi\alpha} \, r 
    \end{aligned}
    \end{equation}
    and thus, $f_x(\eta_r) < f_x(\xi)$ for $r > 0$ sufficiently small, contradicting \eqref{eq:ximin}. 
    Similarly, in the two-dimensional case, \eqref{eq:superharmest} with $\kappa = 1$ and \eqref{eq:p12d} imply that, for all $r < \rho$, 
    \begin{equation}
    \begin{aligned}
        f_x(\eta_r) & \leq \frac{1}{\alpha}\big(p_3(\xi) + \varphi(\xi) + \|\varphi\|_{C^{1,1}(B_\rho(\xi))} \, r^2\big) + \frac{\beta}{8\pi \alpha}\,r^2 \big(\ln(r) - 1 \big)  \\
        &= f_x(\xi) +  \frac{\beta}{8\pi \alpha}\, r^2 \Big( \ln(r) + \frac{8\pi}{\beta} \|\varphi\|_{C^{1,1}(B_\rho(\xi))} - 1 \Big)  .
    \end{aligned}
    \end{equation}
    Again, for $r>0$ sufficiently small, we obtain $f_x(\eta_r) < f_x(\xi)$, which contradicts \eqref{eq:ximin}.
    Consequently, in both cases $d=2$ and $d=3$, $x$ has to be equal to $\xi$, 
    if $(x, \xi) \in \supp(\bar\pi)$ and no mass from any other point can be transported to $\xi$ as claimed.
    Furthermore, \eqref{eq:prop_support} implies
    \begin{equation*}
    \begin{aligned}
        0 < \beta = \bar u(\{\xi\}) 
        & =  \bar\pi( (\omega_0 \times \{\xi\}) \cap \supp(\bar\pi)) \\
        & = \bar\pi( (\{\xi\} \times \{\xi\}) \cap \supp(\bar\pi)) \le \bar\pi( \{\xi\} \times \omega_1)
        = u^0(\{\xi\})    
    \end{aligned}
    \end{equation*}
    and, consequently
    $\{\xi\}$ is also an atom of $u^0$.
\end{proof}

\begin{remark}
    At the end of \cref{sec:metric}, we will state a trivial example, which demonstrates that $\bar u$ may well contain a Dirac in $\interior(\omega)$,
    if $u^0$ does so, which demonstrates that \cref{thm:nodiractransport} is indeed sharp in case of metric transportation costs and \cref{cor:nodiracinterior} 
    does not hold in this case.
\end{remark}


\section{Existence of a transport map for strongly convex costs }\label{sec:stritctconv}

In this section, we consider the case of strongly convex transportation costs. To be more precise, we require the following

\begin{assumption}\label{assu:strongconv}
    We assume that $\omega_1$ is convex. Moreover, let $d_0 = d_1 =: d$ so that $\omega_0 , \omega_1 \subset \R^d$. 
    The transportation costs are given by $c(x, \xi) = h(\xi-x)$ with a function $h : \R^d \to \R$ that is strongly convex on bounded sets, i.e., for all $R > 0$
    there exists $\beta_R > 0$ such that 
    \begin{equation}\label{eq:hstrongconv}
        \dual{ x^* - \xi^* }{x - \xi} \ge \beta_R \norm{x - \xi}^2
        \qquad
        \forall (x, x^*), (\xi, \xi^*) \in \graph(\partial h).
    \end{equation}
\end{assumption}

Note that a finite convex function on $\R^d$ is automatically locally Lipschitz continuous such that 
$c$ defined above satisfies \cref{assu:standing}.

\begin{remark}
    The classical example for \cref{assu:strongconv} is $h(\xi-x) = \frac{1}{\gamma} \|\xi-x\|^\gamma$ with $\gamma \in (1,2]$. 
    Note that $h$ is twice continuously differentiable on $\R^d \setminus \set{0}$
    with
    \begin{equation*}
        \nabla^2 h(x)
        =
        \norm{x}^{\gamma-2} \parens*{ I - (2-\gamma) \frac{x}{\norm{x}} \frac{x^\top}{\norm{x}}}
        \succeq
        (\gamma-1) \norm{x}^{\gamma-2} I
        \succeq
        (\gamma-1) R^{\gamma-2} I
        .
    \end{equation*}
    Thus,
    $\beta_R = (\gamma-1) R^{\gamma-2}$ is possible in \eqref{eq:hstrongconv}.
    This choice of the cost function $c$ gives rise to $\wdist^c_{u^0}(u) = W_\gamma(u^0, u)^\gamma$ so that the generalized transportation distance is nothing else 
    than the Wasserstein distance of order $\gamma$
    raised to the $\gamma$-th power
    between $u^0$ and $u$.
\end{remark}

Since $\omega_0$ and $\omega_1$ are compact by assumption, there exists a radius $\varrho > 0$ such that $\omega_1 - \omega_0 \subset B_\varrho(0)$.
In the following, we employ the strong convexity of $h$ only on $B_\varrho(0)$. 
To ease notation, we denote the modulus of strong convexity on this set simply by $\beta \coloneqq \beta_\varrho$.

\begin{theorem}\label{thm:transmap}
    Let \cref{assu:strongconv} hold 
    and let $\bar u$ and its associated state $\bar y$ be locally optimal for \eqref{eq:optctrl}. Assume moreover 
    that the associated adjoint state $p$ from \eqref{eq:adjpde} is continuously differentiable in $\omega_1$ and that 
    there is a constant $\kappa < \alpha \, \beta$ such that 
    \begin{equation}\label{eq:kruemmungp}
        \dual{\nabla p(\xi) - \nabla p(\zeta)}{\xi - \zeta} \geq - \kappa\, \|\xi - \zeta\|^2 \quad \forall\, \xi, \zeta \in \omega_1.
    \end{equation}
    Then
    the optimal transport from $u^0$ to $\bar u$
    is given by a unique transport plan which is induced by
    a transportation map $T: \omega \to \omega$ such that $T_\# u^0 = \bar u$.
    Moreover, $T$ is H\"older continuous 
    with exponent $1/2$.
    If $h$ is continuously differentiable with Lipschitz continuous gradient in $\omega_0 - \omega_1$, then $T$ is Lipschitz continuous, too.
\end{theorem}
\begin{proof}
   We again employ the complementarity relation of Kantorovich duality from \eqref{eq:optimality_support_2}, which now reads 
    \begin{equation}\label{eq:supppi}
        \supp(\bar\pi)
        \subset
        \big\{(x,\xi)\in \omega_0 \times \omega_1 : h(\xi-x) + \frac{1}{\alpha} p(\xi) = \inf_{\eta \in \omega_1} h(\eta-x) + \frac{1}{\alpha} p(\eta)\big\} ,
    \end{equation}
    where $\bar\pi$ is an arbitary optimal transport plan $\bar\pi \in \KK_c(u^0, \bar u)$.
    Since $h$ and $p$ are continuous and $\omega_1$ is compact, 
    the infimum in \eqref{eq:supppi} is attained for every $x\in \omega_0$. Moreover, thanks to \eqref{eq:kruemmungp}, the strong convexity of $h$, and the convexity of $\omega_1$, 
    the objective $\omega_1 \ni \eta \mapsto h(\eta-x) + \frac{1}{\alpha} p(\eta)$ is strongly convex with constant $\beta - \kappa/\alpha > 0$
    such that the minimizer is unique. 
    Thus the mapping 
    \begin{equation}\label{eq:deftransmap}
        T : x \ni \omega_0 \mapsto \xi = \argmin_{\eta \in \omega_1} h(\eta-x) + \frac{1}{\alpha} p(\eta)
        \in \omega_1
    \end{equation}
    is well defined. Let us investigate the regularity of $T$.
    The strong convexity of the objective in \eqref{eq:deftransmap} implies
    \begin{align*}
        h(\xi_1 - x_1) + \frac{1}{\alpha} p(\xi_1) + \frac{\beta - \kappa/\alpha}{2} \norm{\xi_1 - \xi_2}^2
        &\le
        h(\xi_2 - x_1) + \frac{1}{\alpha} p(\xi_2)
        \\
        h(\xi_2 - x_2) + \frac{1}{\alpha} p(\xi_2) + \frac{\beta - \kappa/\alpha}{2} \norm{\xi_1 - \xi_2}^2
        &\le
        h(\xi_1 - x_2) + \frac{1}{\alpha} p(\xi_1)
        .
    \end{align*}
    Adding these inequalities yields
    \begin{equation}
        \label{eq:fancy_estimate_for_sensitivity}
        \begin{aligned}
            (\beta - \kappa/\alpha) \norm{\xi_2 - \xi_1}^2
            &\le
            h(\xi_2 - x_1)
            -
            h(\xi_2 - x_2)
            +
            h(\xi_1 - x_2)
            -
            h(\xi_1 - x_1)
            \\&
            \le
            2\,\lip_{\omega_1-\omega_0}(h) \norm{x_1 - x_2},
        \end{aligned}
    \end{equation}
    where $\lip_{\omega_1 - \omega_0}(h)$ is the Lipschitz constant of the convex function $h$ on the compact set $\omega_1 - \omega_0$.
    This yields the claimed  Hölder continuity of $T$ with exponent $1/2$.

    Therefore $T$ is Borel measurable and we have to check that it defines a transport map. 
    To see this, we note that \eqref{eq:supppi} and \eqref{eq:deftransmap}
    imply that $\supp(\bar\pi) \subset G := \graph(T)$ for every optimal transport plan.
    This implies $\bar\pi = (\id, T)_{\#} u^0$, since
    \begin{align*}
        \parens*{ (\id, T)_{\#} u^0 }(B)
        &=
        u^0( (\id, T)^{-1} (B) )
        =
        u^0\parens*{
            \set*{
                x \in \omega_0 \given (x, T(x)) \in B
            }
        }
        \\&=
        \bar\pi\parens*{
            \set*{
                x \in \omega_0 \given (x, T(x)) \in B
            }
            \times
            \omega_1
        }
        \\&=
        \bar\pi\parens*{
            \bracks*{
                \set*{
                    x \in \omega_0 \given (x, T(x)) \in B
                }
                \times
                \omega_1
            }
            \cap
            G
        }
        =
        \bar\pi( B \cap G )
        =
        \bar\pi( B )
    \end{align*}
    for every Borel set $B \subset \omega_0 \times \omega_1$.
    Hence, every optimal transport plan is induced by the map $T$
    and this shows uniqueness.
     
    To show the improved regularity of $T$, if $h$ is differentiable, we improve \eqref{eq:fancy_estimate_for_sensitivity}
    by using
    \begin{align*}
        \MoveEqLeft
        h(\xi_2 - x_1)
        -
        h(\xi_2 - x_2)
        +
        h(\xi_1 - x_2)
        -
        h(\xi_1 - x_1)
        \\
        &=
        \int_0^1 \dual{\nabla h(\xi_1 - x_1 + \theta(\xi_2 - \xi_1)) - \nabla h(\xi_1 - x_2 + \theta(\xi_2 - \xi_1)) }{\xi_2 - \xi_1} \,\d\theta
        \\
        &\le
        \lip_{\omega_1 - \omega_0}(\nabla h)\,\norm{x_1 - x_2}\,\norm{\xi_1 - \xi_2}
        .
    \end{align*}
    This implies the Lipschitz continuity of $T$.
\end{proof}

\begin{remark}
    \cref{thm:transmap} implies the existence of an optimal transport map without any assumptions 
    on the marginals $u^0$ and $\bar u$. For the sole Kantorovich problem with fixed (i.e., given) marginals, 
    the existence of a transport map
    from $u^0$ to $\bar u$
    can, in general only be shown, if $\bar u$ is absolutely continuous 
    w.r.t.\ the Lebesgue measure, even in case of strictly convex transportation costs, see \cite[Theorem~1.17]{santambrogio} for details. 
    A classical counterexample is a Dirac measure that is transported to the Lebesgue measure. 
    In the optimal control setting of \cref{thm:transmap} however, the additional regularity of the Kantorovich 
    potential given by the adjoint state allows us to establish the existence of the transport map without any further 
    regularity assumptions on the marginals.
\end{remark}

The existence of a (continuous)
transport map implies several useful consequences.
In particular, the application of \cref{lem:support_pushforward}
implies the next result.

\begin{lemma}\label{lem:suppu}
    If the assumptions of \cref{thm:transmap} are satisfied,
    we have $\supp(\bar u) = T(\supp(u^0))$.
\end{lemma}

By \cref{thm:transmap}, we know that the transport map $T$ is even H\"older continuous with exponent $\theta \in [1/2, 1]$.
Thus \cref{lem:suppu} implies that for every $s\in [0,\infty)$, it holds 
\begin{equation*}
    \HH^{s/\theta}(\supp(\bar u)) \leq \kappa_T^{s/\theta}\, \HH^{s}(\supp(u^0)),
\end{equation*}
where $\HH^s$ denotes the $s$-dimensional Hausdorff measure and $\kappa_T > 0$ is the H\"older constant of $T$. 
Note that this inequality is shown in \cite[Proposition~3.5]{Maggi2012}
for the Lipschitz case $\theta = 1$
and the proof directly generalizes to $\theta < 1$.

Then, if we set $s = \dim_{\HH}(\supp(u^0)) + \varepsilon$ with some $\varepsilon > 0$, the above estimate implies
\begin{equation*}
    \HH^{s/\theta}(\supp(\bar u)) \leq \kappa_T^{s/\theta}\, \HH^{s}(\supp(u^0)) = 0
\end{equation*}     
such that $\dim_{\HH}(\supp(\bar u)) \leq s/\theta$. Letting $\varepsilon \searrow 0$ thus shows that 
the Hausdorff dimension $\dim_{\HH}(\supp(\bar u))$ is at most $\theta^{-1}$ times larger than 
the one of $\supp(u^0)$ under the Assumptions of \cref{thm:transmap}.
As a consequence, we have $\dim_{\HH}(\supp(\bar u)) = 0$ whenever 
$\dim_{\HH}(\supp(u^0)) = 0$. If $u^0$ is a discrete measure, i.e., a weighted sum of countably many 
Dirac measures, we can say even more:

\begin{corollary}\label{cor:barudiracs}
    Let \cref{assu:strongconv} be fulfilled and suppose that the adjoint state fulfills \eqref{eq:kruemmungp}. 
    If $\supp(u^0)$ consists of countable many points, 
    then $\supp(\bar u)$ consists of at most countable many points, too. 
    Similarly, if $\supp(u^0)$ is a finite set with cardinality $m\in \N$, then $|\supp(\bar u)| \leq m$.
\end{corollary}

\begin{proof}
    The assertions follow immediately from \cref{thm:transmap} and \cref{lem:suppu}, respectively.
    If $\supp(u^0) = \bigcup_{i\in\N} x_i$, then by \cref{lem:suppu} there holds
    $\supp(\bar u) = \bigcup_{i\in\N} T(x_i)$, which gives the first claim. The second can be proven analogously.
\end{proof}

The condition in \eqref{eq:kruemmungp} imposes a hard assumption on the curvature and the regularity of the adjoint state.
So the crucial question is, if there are examples were this assumption is met. For this purpose, we assume that 
$\omega_1 \subset \Omega$ and consider again a tracking type objective, this time however with an observation domain $D$ disjoint from $\omega_1$, i.e.,
\begin{equation}\label{eq:tracking}
    J(y) := \frac{1}{2} \, \| y - y_d\|^2_{L^2(D)},
\end{equation}
where $D \subset \Omega$ is a measurable set with $\dist(\overline{D}, \omega_1) > 0$.
Furthermore, $y_d$ is a given desired state in $L^2(D)$.
Due to $y\in W^{1,q}_0(\Omega) \embed L^s(\Omega)$ with $s =  dq/(d-q) \geq 2$,
which follows from \cref{assu:standing}, the objective $J$ is well defined and continuous.
The adjoint equation from \eqref{eq:adjpde} then reads
\begin{equation}\label{eq:adjpdeex}
    - \laplace p = \chi_D ( \bar y - y_d ) \quad \text{in } W^{-1,q'}(\Omega)
\end{equation}
and admits a unique solution $p \in W^{1,q'}_0(\Omega) \subset C(\overline{\Omega})$. 
With a little abuse of notation, we denote the solution operator of state equation with domain $\mathfrak{M}(\overline\Omega)$ 
and values in $L^2(D)$ by $S$, too. Then the solution operator of \eqref{eq:adjpdeex} is just the adjoint thereof. Now, since 
$\dist(\overline{D}, \omega_1) > 0$,
the adjoint state $p$ is harmonic in
$\Omega \setminus \overline{D}$.
We are going to repeat the arguments leading to \cite[Theorem~2.10]{GT01}
in order to get a slightly better estimate.
We define $r := \dist(\omega_1, \partial( \Omega \setminus \overline{D} ))$
and $C := \norm{p}_{C(\overline\Omega)}$.
From \cite[(2.31)]{GT01}, we get
\begin{equation}
  \label{eq:estimate_gradient}
  \norm{\nabla p(x)} \le \frac{2 d}{r} C \qquad\forall x \in \omega_1 + B_{r/2}(0).
\end{equation}
Now, let $v \in \R^d$ with $\norm{v} \le 1$ be arbitrary.
Since the Hessian $\nabla^2 p$ is harmonic in
$\Omega \setminus \overline{D}$,
we get
for all $y \in \omega_1$
the estimate
\begin{align*}
  \nabla^2 p(y) v
  &=
  \frac1{\omega_d (r/2)^d} \int_{B_{r/2}(y)} \nabla^2 p(x) v \, \d\lambda^{d}(x)
  \\&
  =
  \frac1{\omega_d (r/2)^d} \int_{\partial B_{r/2}(y)} (\nabla p(x)^\top v) n \, \d\HH^{d-1}(x)
  \\&
  \le
  \frac{2 d}{r} \sup_{x \in \partial B_{r/2}(y)} \abs{\nabla p(x)^\top v}
  \le
  \frac{4 d^2}{r^2} C,
\end{align*}
where $\omega_d$ is the volume of the unit ball of $\R^d$
and we used (in order):
mean value property of harmonic functions,
Gauß divergence theorem,
simple estimate of the integral
and
\eqref{eq:estimate_gradient}.
Thus,
\begin{equation}\label{eq:giltru}
    \| \nabla^2 p \|_{C(\omega_1)}
    \leq \frac{4d^2}{r^2} \, \|S\|_{\LL(\mathfrak{M}(\overline{\Omega}), L^2(D))}  \, \| \bar y - y_d\|_{L^2(D)}.
\end{equation}
Consequently \eqref{eq:kruemmungp} is satisfied provided that 
\begin{equation}\label{eq:kruemmungp2}
    \frac{4d^2 \, \|S\|_{\LL(\mathfrak{M}(\overline{\Omega}), L^2(D))}}{\dist(\omega_1, \partial( \Omega \setminus \overline{D} ))^2} \, \| \bar y - y_d\|_{L^2(D)} 
    < \alpha\, \beta.
\end{equation}
We thus arrive at the following result:

\begin{corollary}
    Suppose that, in addition to \cref{assu:strongconv}, $\omega_1\subset \Omega$ and
    the objective is given by \eqref{eq:tracking} with $\dist(\overline{D}, \omega_1) > 0$. 
    Let $\bar u \in \mathfrak{M}(\omega_1)$ be an optimal solution of \eqref{eq:optctrl} and suppose that 
    the associated state $\bar y$ satisfies \eqref{eq:kruemmungp2}. 
    Then, the assertions of \cref{thm:transmap} hold true, i.e., there exists a transport map from $u^0$ to $\bar u$, 
    which is H\"older continuous and even Lipschitz continuous, provided that $h\in C^{1,1}(\R^d)$.
\end{corollary}

The condition in \eqref{eq:kruemmungp2} can be checked a priori. In the Wasserstein 2-case for instance, where $h(\xi) = \frac{1}{2} \|\xi\|^2$, 
we have $\beta = 1$. Moreover, the constraints in \eqref{eq:kant} imply
\begin{equation*}
    \|\bar y\|_{L^2(D)} \leq \|S\|_{\LL(\mathfrak{M}(\overline{\Omega}), L^2(D))}\, |u^0|(\omega_0)
\end{equation*}
and thus \eqref{eq:kruemmungp2} is fulfilled, provided that the Tikhonov regularization parameter $\alpha$ satisfies
\begin{equation}\label{eq:tikhonov}
    \alpha > \frac{4\,d^2 \, \|S\|_{\LL(\mathfrak{M}(\overline{\Omega}), L^2(D))}}{\dist(\omega_1, \partial( \Omega \setminus \overline{D} ))^2} \, 
    \Big( \|y_d\|_{L^2(D)} + \|S\|_{\LL(\mathfrak{M}(\overline{\Omega}), L^2(D))}\, |u^0|(\omega_0)\Big).
\end{equation}

\begin{corollary}
    Let \cref{assu:strongconv} hold true and assume that $\omega_1 \subset \Omega$. Let $\wdist^c_{u^0}$ be given by the Wasserstein 2-distance 
    and $J$ be given by the tracking type objective in \eqref{eq:tracking} 
    with $\dist(\overline{D}, \omega_1) > 0$. Suppose moreover that the Tikhonov parameter $\alpha$ fulfills \eqref{eq:tikhonov}. 
    Then, for every solution $\bar u$ of \eqref{eq:optctrl}, there exists a transport map from $u^0$ to $\bar u$, i.e., $\bar u = T_{\#} u^0$, which is Lipschitz 
    continuous such that $\dim_{\HH}(\supp(\bar u)) \leq \dim_{\HH}(\supp(u^0))$.
\end{corollary}

\begin{remark}
    Note that the two above corollaries rely on the assumption that the observation domain $D$ and the control domain $\omega_1$ are separated. 
    In this case however, the theory of \cref{sec:tracking} does not apply, as it is essential for the analysis underlying \cref{sec:tracking} 
    that the observation domain contains the control domain, cf.\ the proof of \cref{lem:yboundsupp}. 
\end{remark}


\section{Power-type cost and absolutely continuous prior}\label{sec:abscont}

Throughout this section, we impose the following

\begin{assumption}\label{assu:u0abscont}
    Beside our standing assumptions, we suppose that
    $d_0 = d_1 =: d$.
    We moreover assume that $u^0 \ll \lambda^d$ and the associated density 
     is essentially bounded, i.e., $U^0 := \frac{\d u^0}{\d \lambda^d} \in L^\infty(\omega_0)$. 
\end{assumption}
Under this assumption, Brenier's theorem, cf.\ \cite[Theorem~1.17]{santambrogio}, yields the existence of an optimal transport map  $T: \omega_0 \to \omega_1$ such that 
$\bar u = T_\# u^0$ for an optimal control $\bar u$.
Instead of using this map $T$,
we will see that
the optimality system from \cref{thm:fon}
(and its consequence \cref{lem:support_plan})
allows us to apply Brenier's method of proof in the other 
direction, too.

We begin with an interior estimate for $\bar u$
which requires some regularity of the adjoint state.

\begin{theorem}\label{thm:uabscont}
    Let \cref{assu:u0abscont} hold true and
    assume that the transportation costs are given by $c(x,\xi) = \frac{1}{\gamma} \|\xi-x\|^\gamma$
    for some $\gamma \in (1,2]$.
    Let $\bar u \in \frakM(\omega_1)$ be an optimal control
    with associated adjoint state $p \in W^{1,q'}_0(\Omega)$.
    We further suppose that $p \in W^{2,r}_{\textup{loc}}(\interior(\omega_1))$ with $r \in (d,\infty]$.
    Then we have
    $\bar u\mres \interior(\omega_1) \ll \lambda^d$.
    The associated density function satisfies
    \begin{equation}\label{eq:densities}
        \bar U_{\interior(\omega_1)} := \frac{\d \bar u \mres \interior(\omega_1)}{\d \lambda^d}
        \in
        L^{r /d}_{\textup{loc}}(\interior(\omega_1))
        .
    \end{equation}
    If $\interior(\omega_1)$ satisfies the cone condition and $p \in W^{2,r}(\interior(\omega_1))$ with $r \in (d,\infty]$, 
    then we even have $\bar U_{\interior(\omega_1)} \in L^{r /d}(\interior(\omega_1))$.
\end{theorem}

\begin{proof}
    We denote by
    $\bar\pi$ the optimal transport plan (which is unique in this case, see \cite[Theorem~1.17]{santambrogio}).  
    We utilize \eqref{eq:optimality_support_2}
    which now reads
    \begin{equation}\label{eq:supppi2}
        \supp(\bar\pi)  \subset \Big\{(x,\xi) \in \omega_0 \times \omega_1 :
        \tfrac{1}{\gamma}\, \|\xi-x\|^\gamma - \psi(\xi) = \inf_{\eta \in \omega_1} \tfrac{1}{\gamma} \, \|\eta-x\|^\gamma - \psi(\eta)\Big\},
    \end{equation}
    where $\psi$ is defined by $\psi = - p /\alpha$.
    Let now $(x, \xi) \in \supp(\bar\pi)$ with $\xi\in \interior(\omega_1)$ be arbitrary. 
    Then, due to $p \in W^{2,r}_{\textup{loc}}(\interior(\omega_1)) \embed C^1(\interior(\omega_1))$, the necessary optimality conditions for the minimization problem in \eqref{eq:supppi2} read
    \begin{equation*}
        0 = \norm{\xi - x}^{\gamma - 2} (\xi - x) - \nabla\psi(\xi).
    \end{equation*}
    Rearranging yields
     \begin{equation}\label{eq:deftildeT}
         x = \xi - \norm{\nabla\psi(\xi)}^{\frac{2-\gamma}{\gamma-1}} \nabla\psi(\xi) =: \widetilde T(\xi)
         \quad \text{for all } (x, \xi) \in \supp(\bar\pi) \cap (\omega_0 \times \interior(\omega_1)).
    \end{equation}
    In order to show that $\bar u$ is absolutely continuous in the interior of $\omega_1$,
    let a Borel set $A \subset \interior(\omega_1)$ be given.
    Since $\bar u$ is the marginal of $\bar \pi$, we have
    \begin{equation*}
        \bar u(A)
        =
        \bar\pi( \omega_0 \times A )
        =
        \bar\pi( \supp(\bar\pi) \cap (\omega_0 \times A) )
        ,
    \end{equation*}
    where we also used \eqref{eq:prop_support}.
    From \eqref{eq:deftildeT}
    we get
    \begin{equation*}
        \supp(\bar\pi) \cap (\omega_0 \times A)
        \subset
        \set{ (\widetilde T(\xi), \xi) \given \xi \in A }
        \subset
        \widetilde T(A) \times \omega_1
        .
    \end{equation*}
    Together with the previous equality,
    this yields
    \begin{equation}
        \label{eq:estimate_bar_u_A}
        \bar u(A)
        \le
        \bar\pi( \widetilde T(A) \times \omega_1 )
        =
        u^0( \widetilde T(A) )
        =
        \int_{\widetilde T(A)} U^0(x)\,\d \lambda^d(x)
        .
    \end{equation}
    Now let $K \subset \R^d$ be an arbitrary compact subset of $\interior(\omega_1)$. Then, by compactness, there exists 
    a finite cover of $K$ of open balls and, w.l.o.g., we may assume that the radii of these balls are so small so that the closure of each such ball 
    is contained in $\interior(\omega_1)$. Now let $B$ be one of these balls. Then, due to $p \in W^{2,r}_{\textup{loc}}(\interior(\omega_1))$, 
    we have $\nabla \psi \in W^{1,r}(B)$ and, since $B$ has a regular boundary, this in turn implies  $\nabla \psi \in L^\infty(B)$ thanks to to $r > d$.
    Due to $\gamma \in (1,2]$, the function
    $\R^d \ni v \mapsto \norm{v}^{(2-\gamma)/(\gamma-1)} v \in \R^d$ is Lipschitz on bounded sets.
    Consequently, the chain rule from \cite[Theorem~1]{MarcusMizel1979} gives $\widetilde T \in W^{1,r}(B;\R^d)$,
    where we again used that $B$ admits a regular boundary and thus satisfies the cone condition. 
    Since $B$ was an arbitrary ball from the finite cover of $K$ and $K$ was an arbitrary compact subset of $\interior(\omega_1)$, 
    this implies $\widetilde T \in W^{1,r}_{\textup{loc}}(\interior(\omega_1);\R^d)$.
    By assumption, $r > d$,
    which implies that $\widetilde T$ is continuous on $\interior(\omega_1)$
    and Hölder continuous in a neighborhood of every point from $\interior(\omega_1)$.
    Consequently, we can apply the change-of-variables formula
    from \cite[Prop.~1.1 and Thm.~1.3]{Maly1994}.
    Therein, the multiplicity function
    \begin{equation*}
        N(x, \widetilde T, A)
        =
        \#\set{ \xi \in A \given \widetilde T(\xi) = x }
    \end{equation*}
    appears and we trivially have
    $N(x, \widetilde T, A) \ge 1$ for all $x \in \widetilde T(A)$.
    Consequently,
    \begin{align*}
        \int_{\widetilde T(A)} U^0(x)\,\d \lambda^d(x)
        &\le
        \int_{\widetilde T(A)} N(x, \widetilde T, A) U^0(x)\,\d \lambda^d(x)
        \\&=
        \int_{A} U^0(\widetilde T(\xi)) |\det D \widetilde T(\xi)|\,\d \lambda^d(\xi) 
        .
    \end{align*}
    Together with \eqref{eq:estimate_bar_u_A},
    we get
    \begin{equation}\label{eq:area}
        \bar u (A) \leq \int_{A} U^0(\widetilde T(\xi)) |\det D \widetilde T(\xi)|\,\d \lambda^d(\xi) 
        \quad \forall\, A \in \BB(\interior(\omega_1)) .
    \end{equation}
    First of all, this tells us that indeed $\bar u \mres \interior(\omega_1) \ll \lambda^d$ as claimed.
    There is thus a density function 
    $\bar U_{\interior(\omega_1)} := \frac{\d \bar u \mmres \interior(\omega_1)}{\d \lambda^d}\in L^1(\interior(\omega_1))$. 
    Since \eqref{eq:area} holds for every Borel subset of $\interior(\omega_1)$, we deduce 
    \begin{equation}\label{eq:barUpointwise}
        0 \leq \bar U_{\interior(\omega_1)} \leq (U^0 \circ \widetilde T) |\det D \widetilde T| \quad \text{$\lambda^d$-a.e.\ in $\interior(\omega_1)$},
    \end{equation}        
    where the non-negativity of $\bar U$ follows from $\bar u \geq 0$. This implies
    \begin{align*}
        |\bar{U}_{\interior(\omega_1)}(\xi)|^{r /d}
        & \le |U^0(\widetilde T(\xi))|^{r  /d} \, |\det D\widetilde T(\xi)|^{r /d} \\ 
        & \le C \, \|U^0\|_{L^\infty(\omega_0)}^{r  /d} \, |D\widetilde T(\xi)|^{r}
    \end{align*}
    for $\lambda^d$-a.a.\ $\xi \in \interior(\omega_1)$.
    This shows the regularity claims for
    $\bar U_{\interior(\omega_1)}$. 
    Note that for the global regularity result on $\interior(\omega_1)$, one needs that $\interior(\omega_1)$ satisfies the cone condition 
    in order to apply the chain rule and Sobolev embeddings on that set as we did for the compact subset $K$ above.
\end{proof}

For boundary estimates, we need the convexity of the domain $\Omega$
and an auxiliary lemma.
\begin{lemma}
	\label{lem:convex_set_stuff}
	Let $\Omega \subset \R^d$, $d \in \N$, be convex, open and bounded.
	Let a Borel set $A \subset \partial\Omega$ be given
	and define
	\begin{equation*}
		B
		:=
		\set{
			x \in \overline\Omega
			\given
			\exists y \in A:
			\dist(x, \partial\Omega) = \norm{x - y}
		}
		.
	\end{equation*}
	Then,
	$B$ is Lebesgue measurable and
	\begin{equation*}
		\lambda^d(B)
		\le
		r \HH^{d-1}(A)
		,
	\end{equation*}
	where $r := \max_{x \in \Omega} \dist(x, \partial\Omega)$.
\end{lemma}
\begin{proof}
	We define the set-valued map
	$F \colon \overline\Omega \rightrightarrows \partial\Omega$,
	\begin{equation*}
		F(x) := \set{ y \in \partial\Omega \given \dist(x, \partial\Omega) = \norm{x - y} }.
	\end{equation*}
	Then, it is clear that $F$ has nonempty and closed images.
	Its graph
	\begin{equation*}
		\graph(F) =
		\set{ (x, y) \in \overline\Omega \times \partial\Omega \given \dist(x, \partial\Omega) = \norm{x - y} }
	\end{equation*}
	is a closed subset of $\overline\Omega \times \partial\Omega$,
	consequently, it is Borel.
	Now, we can apply
	\cite[Theorem~8.1.4]{AubinFrankowska2009}
	to get the measurability of $F$
	and, in particular,
	$B = \set{x \in \overline\Omega \given F(x) \cap A \ne \emptyset } = F^{-1}(A)$
	is a Lebesgue measurable subset of $\overline\Omega$.

	We denote by $s \colon \R^d \to \R$ the signed distance function of $\Omega$.
	It is easy to check that $s$ is convex and has Lipschitz constant $1$.
	Rademacher's theorem implies that $s$ is differentiable a.e.\ and
	from \cite[Theorem~4.8]{Federer1959}
	we get $\norm{\nabla s} = 1$ at all points of differentiability.
	Consequently, we can apply the coarea formula
	from \cite[Theorem~18.1]{Maggi2012}
	to obtain
	\begin{equation*}
		\lambda^d(B)
		=
		\int_B \norm{\nabla s} \, \d\lambda^d
		=
		\int_{\R} \HH^{d-1}( B \cap \set{s = t}) \, \d t
		=
		\int_{-r}^0 \HH^{d-1}( B \cap \set{s = t}) \, \d t
		,
	\end{equation*}
	where the last identity uses that $s$ takes only values in $[-r,0]$ on $\overline\Omega$.
		Note that the coarea formula in \cite{Maggi2012}
		is only stated for Borel sets.
		However, we can approximate the Lebesgue measurable set $B$
		with Borel sets $B_1, B_2$ such that
		$B_1 \subset B \subset B_2$
		and
		$\lambda^d(B_1) = \lambda^d(B) = \lambda^d(B_2)$.
		An application of the Borelian coarea formula to $B_1$
		and $B_2$ yields the above formula for $B$.

	Now, for any $t \in (-r,0)$,
	the set $\set{s \le t}$ is convex and closed.
	We claim
	\begin{equation}
		\label{eq:some_sets}
		B \cap \set{s = t}
		\subset
		\proj_{\set{s \le t}}(A)
		.
	\end{equation}
	Let $x \in B \cap \set{s = t}$ be given,
	i.e., $\dist(x, \partial\Omega) = -t$.
	The definition of $B$ implies the existence of $y \in A$
	with $\norm{x - y} = -t$.
	Set $z := \proj_{\set{s \le t}}(y)$.
	Then, $\norm{z - y} \le -t$, since $\norm{x - y} = -t$.
	Moreover,
	$0 = s(y) \le s(z) + \norm{z - y} \le t - t = 0$.
	This shows $\norm{z - y} = -t$
	and, thus, $x = z \in \proj_{\set{s \le t}}(A)$.
	Hence, \eqref{eq:some_sets} holds.

	Now, since $\proj_{\set{s \le t}}$ is Lipschitz with constant $1$,
	we get
	\begin{equation*}
		\HH^{d-1}(B \cap \set{s = t})
		\le
		\HH^{d-1}(A),
	\end{equation*}
	see
	\cite[Proposition~3.5]{Maggi2012}.
	Using this estimate in the coarea formula above yields the claim.
\end{proof}

In the case
$\omega_0 \subset \omega_1 = \overline\Omega$,
this lemma can be used to prove the regularity of the boundary part of
locally optimal controls.

\begin{theorem}\label{thm:uabscont_bdry}
    Let \cref{assu:u0abscont} hold true and suppose that
    $\Omega$ is convex with boundary $\Gamma := \partial\Omega$ and
    $\omega_0 \subset \omega_1 = \overline\Omega$.
    We further assume that the transport cost is given by $c(x,\xi) = h(\norm{x - \xi})$
    for some strictly increasing function $h$.
    Consider a locally optimal control $\bar u \in \frakM(\overline{\Omega})$
    together with the associated adjoint state $p \in W_0^{1,q'}(\Omega) \embeds C_0(\Omega)$.
    Then,
    $\bar u \mres \Gamma \ll \HH^{d-1}$.
    The associated density function satisfies
    \begin{equation}\label{eq:densities}
        \bar U_\Gamma :=  \frac{\d \bar u \mres \Gamma}{\d \HH^{d-1}} \in L^{\infty}(\Gamma;\HH^{d-1}) .
    \end{equation}
\end{theorem}
\begin{proof}
    Next take an arbitrary $(x, \xi) \in \supp(\bar\pi)$ with $\xi \in \Gamma$.
    As in the proof of \cref{thm:uabscont},
    we utilize
    \eqref{eq:optimality_support_2}.
    Using $p = 0$ on $\Gamma$
    this implies
    \begin{equation*}
        h(\norm{\xi - x}) \le \inf_{\eta \in \Gamma} h(\norm{\eta - x}),
    \end{equation*}
    i.e.,
    $\dist(x, \Gamma) = \norm{\xi - x}$, since $h$ is strictly increasing by assumption.
    For a Borel subset $A \subset \Gamma$,
    we define
    \begin{equation*}
        B
        :=
        \set{
            x \in \overline\Omega
            \given
            \exists y \in A:
            \dist(x, \Gamma) = \norm{x - y}
        }
        .
    \end{equation*}
    From \cref{lem:convex_set_stuff},
    we get the Lebesgue measurability of $B$
    and $\lambda^d(B) \le r \HH^{d-1}(A)$
    for some $r > 0$.
    Consequently, there is a Borel set $\widetilde B \subset \overline\Omega$
    with $B \subset \widetilde B$ and $\lambda^d(\widetilde B) = \lambda^d(B) $.
    The above argument implies
    \begin{equation*}
        (\overline\Omega \times A) \cap \supp(\bar\pi)
        \subset
        B \times A
        \subset
        \widetilde B \times A
        .
    \end{equation*}
    Consequently,
    \begin{align*}
        \bar u(A)
        &=
        \bar\pi( \overline\Omega \times A )
        =
        \bar\pi( (\overline\Omega \times A) \cap \supp(\bar\pi) )
        \le
        \bar\pi( \widetilde B \times A )
        \le
        \bar\pi( \widetilde B \times \overline\Omega )
        =
        u^0( \widetilde B )
        \\&
        \le
        r \norm{U^0}_{L^\infty(\Omega)} \HH^{d-1}(A)
    \end{align*}
    for all Borel sets $A \subset \Gamma$.
    An application of the Radon--Nikodým theorem
   (using that $\HH^{d-1} \mres \Gamma$ is a finite measure)
    yields the claim.
\end{proof}

We end this section by applying the above two theorems to a problem with tracking type objective, where  
our findings from \cref{sec:tracking} allow us to verify the regularity assumptions on the adjoint state $p$ 
from \ref{thm:uabscont}:

\begin{proposition}
    Suppose  that  $\Omega$ is convex and  $\omega_0 \subset \omega_1 = \overline\Omega$
    and that $u_0$ has an essentially bounded density w.r.t.\ the Lebesgue measure, see \cref{assu:u0abscont}.
    Further, consider the transport costs $c(x, \xi) = \frac12 \norm{x - \xi}^2$ and
    that the tracking type objective $J$ from \eqref{eq:trackingOmega} so that 
    the optimal control problem reads
    \begin{equation}\tag{P$_{2, \Omega}$}
        \left\{\quad
        \begin{aligned}
            \min \quad & \frac{1}{2}\,\| y - y_d\|_{L^2(\Omega)}^2 + \frac{\alpha}{2}\, W_2(u^0, u)^2\\
            \textup{w.r.t.} \quad & y \in W_0^{1,q}(\Omega), \quad u \in \frakM(\overline{\Omega}),\\
            \textup{s.t.} \quad & - \laplace y = u \mres \Omega \text{ in } W^{-1,q}(\Omega)
        \end{aligned}
        \right.
    \end{equation}
    with a desired state $y_d \in C(\overline{\Omega})$.
    Then every optimal control $\bar u$ is absolutely continuous w.r.t.\ the Lebesgue measure in $\Omega$ and 
    the $(d-1)$-dimensional Hausdorff measure on $\Gamma := \partial\Omega$.
    Moreover, the density function of $\bar u \mres \Gamma$  satisfies \eqref{eq:densities} and the density function 
    of $\bar u \mres \Omega$ is the same for all optimal controls. Furthermore, 
    provided that \cmchange{$y_d \in C^{0,\gamma}(\Omega)$ with some $\gamma > 0$}, 
    the latter satisfies $\bar U _\Omega \in L^\infty_{\textup{loc}}(\Omega)$ and, 
    if, in addition, the domain is of class \cmchange{$C^{2,\gamma}$}, then even $\bar U_\Omega \in L^\infty(\Omega)$.
\end{proposition}

\begin{proof}
    Let us first consider $\bar u \mres \Gamma$. From \cref{thm:fon}, we know that the adjoint state satisfies $p \in C_0(\Omega)$ such that, 
    by setting $h(r) := \frac{1}{2} r^2$, all assumptions of \cref{thm:uabscont_bdry} are fulfilled implying in turn the assertions on $\bar u \mres \Gamma$. 
    
    Concerning $\bar u \mres \Omega$, we apply \cref{thm:uabscont}, for which we have to verify the regularity assumptions 
    on the adjoint state as solution of  
    \begin{equation}\label{eq:adjpdeex}
        - \laplace p = \bar y - y_d \quad \text{in } W^{-1,q'}(\Omega).
    \end{equation}
    Since the transportation costs are quadratic and thus smooth and the desired state is assumed to be continuous, 
    the assumptions of \cref{thm:statebound} are fulfilled and thus the associated state $\bar y$ is essentially bounded in $\Omega$.   
    Therefore, by applying interior regularity results from \cite[Theorem~3.8]{Troianiello1987}, we find
    $p \in W^{2,r}_{\textup{loc}}(\Omega)$ for every $r < \infty$ such that \cref{thm:uabscont} gives that $\bar u \mres \Omega$ is indeed 
    absolutely continuous w.r.t.\ the Lebesgue measure. Moreover, by \cref{cor:convex}, the optimal control restricted to $\Omega$ is unique, 
    which gives the uniqueness of the density $\bar U_\Omega$. 
    
    \cmchange{In order to show the essential boundedness of $\bar U_\Omega$ in the interior, 
    let $\Omega' \subset \R^d$ be an arbitrary open set with $\overline{\Omega'} \subset \Omega$.
    Since $p \in W^{2,r}_{\textup{loc}}(\Omega)$ with $r \in (d,\infty)$ arbitrary, 
    we already know that $\bar U_\Omega \in L^{r/d}_{\textup{loc}}(\Omega)$. 
    Therefore, by applying \cite[Theorem~3.8]{Troianiello1987} to the state equation, we obtain $\bar y \in W^{2,r/d}_{\textup{loc}}(\Omega)$. 
    Hence $\bar y \in C^{0,\gamma}(\Omega'')$, where $\Omega''$ is an open set with smooth boundary and
    $\overline{\Omega'} \subset \Omega'' \subset \overline{\Omega''} \subset \Omega$. 
    Since $y_d \in C^{0,\gamma}(\Omega'')$ by assumption, too, and $p$ is continuous, \cite[Theorem~4.3]{GT01} implies that 
    \begin{equation*}
        - \laplace v = \bar y - y_d \text{ in } \Omega'', \quad v = p \text{ on } \partial\Omega''
    \end{equation*}        
    admits a unique classical solution $v \in C^2(\Omega'')$. Thus we obtain $p \in C^{2}(\overline{\Omega'})$ and, 
    since $\Omega'$ was arbitrary, this implies $p \in W^{2,\infty}_{\textup{loc}}(\Omega)$.
    Hence, $\bar U_\Omega \in L^\infty_{\textup{loc}}(\Omega)$ thanks to \cref{thm:uabscont}.}    
    
    
    For the regularity up to the boundary, we first note that, thanks to the assumed regularity of the domain, 
    \cite[Theorem~9.15]{GT01} implies that the adjoint state $p$ as solution of \eqref{eq:adjpdeex}
    satisfies $p\in W^{2,r}(\Omega)$ for every $r < \infty$. Therefore, \cref{thm:uabscont} yields 
    $\bar U_\Omega \in L^{r/d}(\Omega)$ and thus, again by \cite[Theorem~9.15]{GT01}, 
    $\bar y \in W^{2,r/d}(\Omega) \embed C^{0,\gamma}(\Omega)$.
    \cmchange{Since also $y_d \in C^{0,\gamma}(\Omega)$ 
    and $\Omega$ is of class $C^{2,\gamma}$, 
    \cite[Theorem~6.14]{GT01} implies that $p \in C^{2,\gamma}(\Omega) \embed W^{2,\infty}(\Omega)$.}
    Given this regularity of the adjoint state, \cref{thm:uabscont} directly implies $\bar U_\Omega \in L^\infty(\Omega)$, 
    provided that the domain is of class \cmchange{$C^{2,\gamma}$}, as claimed.
\end{proof}

\begin{remark}
  \label{rem:numerics}
  We mention that the regularity results
  of this section are confirmed
  by the numerical results in \cite{BorchardWachsmuth2025:1}.
  In particular,
  in case $\omega_1 = \overline\Omega$
  the computed approximations of the optimal controls
  seem to be absolutely continuous in the interior with bounded density,
  see \cref{thm:uabscont}.
\end{remark}


\section{Metric transport costs}\label{sec:metric}

As in the previous section,
we assume that the dimension of $\omega_0$ is the same as the one of $\omega_1$, i.e., $\omega_0, \omega_1 \subset \R^d$ with 
$d \coloneqq d_1 = d_0 \ge 2$. This section is devoted to the case, where the transportation costs are given by the 
euclidean distance of two points, i.e., $c(x,\xi) \coloneqq \|x-\xi\|$, such that $\wdist^c_{u^0}(u) = W_1(u^0, u)$ is just the Wasserstein distance of order 1 between $u^0$ and $u$.

\begin{theorem}\label{thm:transray}
    Let $c(x,\xi) \coloneqq \|x-\xi\|$ and let $\bar u$ be an arbitrary local minimizer of \eqref{eq:optctrl}. 
    Moreover, assume that the adjoint state $p$ associated with $\bar u$ is differentiable at every point in $\supp(\bar u) \cap \interior(\omega_1)$. 
    Then, for every $\xi \in \supp(\bar u) \cap \interior(\omega_1)$, there exists an $r \geq 0$ such that 
    \begin{equation}\label{eq:transrayxi}
        \{ x \in \omega_0 : (x, \xi) \in \supp(\bar\pi)\} \subset [\xi, \xi + r \, \nabla p(\xi)],
    \end{equation}
    and the following complementarity system is fulfilled
    \begin{equation}\label{eq:complnablap}
        r \geq 0, \quad r \big(\|\nabla p(\xi)\| - \alpha\big) = 0, \quad \|\nabla p(\xi)\| \leq \alpha .
    \end{equation}
    Thus, if $u^0$ is absolutely continuous w.r.t.\ the Lebesgue measure, then $\bar u$ is non-atomic in the interior of $\omega_1$.

    If, in addition, $\omega_0 \subset \omega_1 = \overline{\Omega}$ and $\Omega$ is convex, then $\bar u \mres \Gamma \ll \HH^{d-1}$ and 
    its density function satisfies $\frac{\d \bar u {\footnotesize \mres} \Gamma}{\d \HH^{d-1}} \in L^\infty(\Gamma;\HH^{d-1})$, where $\Gamma$ again denotes the 
    boundary of $\Omega$.   As a consequence, under these additional assumptions, $\bar u$ is non-atomic on the whole 
    $\omega_1= \overline{\Omega}$, provided that $u^0 \ll \lambda^d$.
\end{theorem}

\begin{remark}
    The set on the right hand side of \eqref{eq:transrayxi} is also known as  \emph{transport rays}, cf.\ \cite[Section~3.1.3]{santambrogio}. 
\end{remark}

\begin{proof}[Proof of \cref{thm:transray}]
  To prove \eqref{eq:transrayxi}, let $\xi \in \supp(\bar u) \cap \interior(\omega_1)$ 
  be arbitrary and employ again the characterization from \cref{lem:support_plan} of the support of $\bar\pi$, 
   which, in case of metric costs, reads
    \begin{equation}
        \label{eq:wo_ist_support_metrisch}
        \supp(\bar\pi) \subset \big\{ (x,\xi) \in \omega_0 \times \omega_1 \colon 
        \|\xi - x\| - \psi(\xi) = \inf_{\eta \in \omega_1} \|\eta - x\| - \psi(\eta) \big\}.
    \end{equation}
    Since $\xi$ belongs to the interior of $\omega_1$, the assumed differentiability of $\psi = - p/\alpha$ and the convexity of the norm imply
    \begin{equation*}
        (x, \xi) \in \supp(\bar\pi) 
        \quad  \Longrightarrow \quad 
        \nabla \psi(\xi) \in \partial  \|\cdot - x\|(\xi).
    \end{equation*}
    Therefore,
    for every $(x, \xi) \in \supp(\bar\pi) \cap (\omega_0 \times \interior(\omega_1))$, there holds
    \begin{equation}\label{eq:transportray}
        x = \xi  - \|\xi - x\| \,\nabla \psi(\xi),
    \end{equation}
    which is \eqref{eq:transrayxi}. From $\nabla \psi(\xi) \in  \partial  \|\cdot - x\|(\xi)$, we moreover deduce that $\|\nabla \psi(\xi)\| \leq 1$ and 
    that $x \neq \xi$ implies $\|\nabla \psi(\xi)\| = 1$. In view of $\psi = - p/\alpha$, this gives \eqref{eq:complnablap}.
    
    Thanks to \eqref{eq:transrayxi}, $\{x \in \omega_0 : (x, \xi) \in \supp(\bar\pi)\}$ 
    is contained in a line segment $[\xi, \hat x]$ with some $\hat x \in \omega_0$. Thus, if $u^0 \ll \lambda^d$, then it follows
    \begin{equation*}
        \bar u(\{\xi\}) = \bar\pi(\omega_0 \times \{\xi\}) 
        \subset  \bar\pi([\xi, \hat x] \times \omega_1) = u^0([\xi, \hat x]) = 0
    \end{equation*}
    so that $\bar u$ is indeed non-atomic in $\interior(\omega_1)$, provided that $u^0 \ll \lambda^d$.
    
    The second assertion immediately follows from \cref{thm:uabscont_bdry} by setting $h$ to the identity in that theorem.
\end{proof}

Under an additional smoothness assumption on the adjoint state, we can prove that $\bar u$ is not only non-atomic but even 
absolutely continuous w.r.t.\ the $(d-1)$-dimensional Hausdorff measure.

\begin{theorem}\label{thm:metrichausdorff}
    Let again $c(x,\xi) \coloneqq \|x-\xi\|$ and let $\bar u$ be an arbitrary local minimizer of \eqref{eq:optctrl} with adjoint state $p$. 
    Suppose moreover that $p$ is continuously differentiable in $\interior(\omega_1)$ with Lipschitz continuous gradient with Lipschitz constant $L > 0$
    and the prior satisfies \cref{assu:u0abscont}, i.e., $u^0$ is absolutely continuous w.r.t.\ the Lebesgue measure with 
    essentially bounded density function $U^0 \in L^\infty(\omega_0)$. Then $\bar u  \mres \interior(\omega_1)$ 
    is absolutely continuous w.r.t.\ $\HH^{d-1}$.
\end{theorem}

\begin{proof} 
First,
$p$ is differentiable at every point in $\supp(\bar u) \cap \interior(\omega_1)$ and \cref{thm:transray} is applicable.
Let us set $T := \alpha^{-1} \sup\set{ \abs{x - \xi} \given x \in \omega_0, \xi \in \omega_1}$
and introduce the function $F \colon \omega_1 \times [0,T] \to \R^d$ by
\begin{equation*}
	F(\xi, t) := \xi + t \nabla p(\xi).
\end{equation*}
Then, \eqref{eq:transrayxi} and \eqref{eq:complnablap}
imply for any Borel set $A \subset \interior(\omega_1)$ with $\HH^{d-1}(A) < \infty$
the inclusion
\begin{align*}
	[\omega_0 \times A] \cap \supp(\bar\pi)
	&\subset
	\set{ (x,\xi) \in \omega_0 \times A \given \exists t \in [0,T] : x = \xi + t \nabla p(\xi) = F(\xi, t)}
	\\
	&\subset
	[ \omega_0 \cap F(A \times [0,T])] \times A
	.
\end{align*}
Consequently,
\begin{equation}\label{eq:coareaineq0}
\begin{aligned}
	\bar u(A)
	&=
	\bar\pi( \omega_0 \times A )
	=
	\bar\pi( (\omega_0 \times A) \cap \supp(\bar\pi) )
	\\
	&\le
	\bar\pi(
		(\omega_0 \cap F(A \times [0,T])) \times \omega_1
	)
	=
	u^0(\omega_0 \cap F(A \times [0,T]))
	\\
	&=
	\int_{\omega_0 \cap F(A \times [0,T])} U^0 \d\lambda^d
	\le
	\norm{U^0}_{L^\infty(\omega_0)} 
	\int_{\omega_0 \cap F(A \times [0,T])} 1 \d\lambda^d
	\\
	&\le
	\norm{U^0}_{L^\infty(\omega_0)} 
	\int_{\omega_0 \cap F(A \times [0,T])} \HH^0( F^{-1}(\set{x}) \cap (A \times [0,T]) ) \d\lambda^d(x)
	.
\end{aligned}
\end{equation}
In the last inequality,
we used that
$x \in \omega_0 \cap F(A \times [0,T])$
implies
$F^{-1}(\set{x}) \cap (A \times [0,T]) \ne \emptyset$,
i.e.,
$\HH^0(F^{-1}(\set{x}) \cap (A \times [0,T]) ) \ge 1$.
We note that, at this point, the measurability of the integrand is not clear, but it will follow from the arguments below.
Next,
we apply the
Nöbeling--Szpilrajn--Eilenberg--Federer--Davies inequality
(or simply coarea inequality)
from \cite[Theorem~1.1]{EsmayliHajlasz2021}
in the setting
\begin{equation*}
	X = \omega_1 \times [0,T]
	,\quad
	Y = \R^d
	,\quad
	t = s = d
	,\quad
	E = A \times [0,T]
	.
\end{equation*}
Note that, due to our assumptions on $p$, the mapping $F$ is Lipschitz continuous with Lipschitz constant 
$L_F \leq 1 + T \, L + \|\nabla p\|_{C(\omega_1)}$
so that the assumptions of \cite[Theorem~1.1]{EsmayliHajlasz2021} are met.
This implies
\begin{equation}\label{eq:coareaineq1}
	\int_{\R^d} \HH^0( F^{-1}(\set{x}) \cap (A \times [0,T]) ) \d\HH^d(x)
	\le
	L_F^d\, \HH^d( A \times [0,T] )
	.
\end{equation}
To estimate the Hausdorff measure on the right hand side, let $\delta > 0$ be arbitrary and consider
\begin{equation*}
\begin{aligned}
    \HH^d_\delta( A \times [0,T] )
    &= \inf \Big\{ \sum_{i\in \N} \frac{\omega_d}{2^d}\, \diam(B_i)^d \colon \diam(B_i) < \delta, \; A \times [0,T] \subset \bigcup_{i\in\N} B_i \Big\} .
\end{aligned}
\end{equation*}
Let $\bigcup_{i\in \N} A_i$ be an arbitrary countable covering of $A$ with $a_i := \diam(A_i) < \delta / \sqrt{2}$.
Set $m_i := \min\{ k\in \N : k a_i > T\}$. Then 
\begin{equation*}
    \bigcup_{i\in \N} \bigcup_{k = 0}^{m_i} \hat B_{ik} \quad \text{with} \quad \hat B_{ik}:= A_i \times [k a_i, (k+1) a_i]
\end{equation*}
is a countable covering of $A \times [0,T]$ with $\diam(\hat B_{ik}) \leq \sqrt{2} a_i < \delta$. Thus we obtain
\begin{equation*}
\begin{aligned}
    \HH^d_\delta( A \times [0,T] )
    & \leq \sum_{i\in \N} \sum_{k=0}^{m_i}  \frac{\omega_d}{2^d}\, \diam(\hat B_{ik})^d \\
    & \leq 2^{d/2-1} \,\frac{\omega_d}{\omega_{d-1}} (T+\delta) \sum_{i\in \N} \frac{\omega_{d-1}}{2^{d-1}}\, \diam(A_i)^{d-1} .
\end{aligned}
\end{equation*}
Since $\bigcup_{i\in \N} A_i$ was an arbitrary covering of $A$ of diameter less than $\delta/\sqrt{2}$, this shows 
$\HH^d_\delta( A \times [0,T] ) \leq 2^{d/2-1} \frac{\omega_d}{\omega_{d-1}} (T+\delta) \, \HH^{d-1}_{\delta/\sqrt{2}}(A)$ and, as $\delta > 0$ was arbitrary, this in turn gives 
\begin{equation}\label{eq:hausdorffprodukt0}
    \HH^d( A \times [0,T] ) \leq 2^{d/2-1} \,\frac{\omega_d}{\omega_{d-1}} \,T \, \HH^{d-1}(A) < \infty.
\end{equation}
Note that, due to \eqref{eq:hausdorffprodukt0}, the integrand on the left hand side of \eqref{eq:coareaineq1} is $\HH^d$-measurable 
so that the integral is well defined, cf.~\cite[Theorem~1.1]{EsmayliHajlasz2021}.
Further, we recall that $\lambda^d = \HH^d$ on $\R^d$.
Returning to \eqref{eq:coareaineq0} and \eqref{eq:coareaineq1}, the inequality in \eqref{eq:hausdorffprodukt0} implies 
\begin{equation}\label{eq:hausdorffprodukt}
	\bar u(A)
	\le
	2^{d/2-1} \,\frac{\omega_d}{\omega_{d-1}}  \,\norm{U^0}_{L^\infty(\omega_0)} \,
	L_F^d\, T \,
	\HH^{d-1}( A )
\end{equation}
for any Borel set $A \subset \interior(\omega_1)$ with $\HH^{d-1}(A) < \infty$.
This shows that $\bar u  \mres \interior(\omega_1)$ is indeed absolutely continuous w.r.t.\ $\HH^{d-1}$.
\end{proof}

Note that we cannot apply the Radon--Nikodým theorem
in the above situation, since the measure $\HH^{d-1} \mres \interior(\omega_1)$
is not $\sigma$-finite (unless $\interior(\omega_1) = \emptyset$).
To give a simple example that the Radon--Nikodým theorem can fail in this case,
we mention that,
similarly to \eqref{eq:hausdorffprodukt}, it holds $\lambda^d(A) \le \HH^{d-1}(A)$ for all Borel sets $A \subset \R^d$,
but, of course, $\lambda^d$ does not have a density w.r.t.\ $\HH^{d-1}$.

In contrast to the Wasserstein-2-case in \cref{thm:uabscont}, we observe that, in case of metric transport costs, we are only able to prove
that the optimal control is absolutely continuous w.r.t.\ $\HH^{d-1}$ and not w.r.t.\ $\lambda^d$, 
provided that the prior is absolutely continuous w.r.t.\ to the Lebesgue measure. The next example shows that 
this result is sharp in the sense that it is indeed possible to obtain an optimal control that is not absolutely continuous w.r.t.\ the Lebesgue measure, 
although the prior is so.

\begin{example}\label{ex:linemeasure}
    We set $\omega_1 = \omega_0 = \overline{\Omega} = \overline{B_1(0)}$ and $\alpha = 1$. For the objective, we choose the tracking type 
    objective from \eqref{eq:trackingOmega}.  Moreover, we set the prior and the optimal control to
    \begin{equation*}
        u^0 := \lambda^2 \mres (B_1(0)\setminus B_{1/2}(0)), \quad \bar u := \frac{3}{4}\, \HH^1 \mres \partial B_{1/2}(0)
    \end{equation*}
    such that
    $u^0(\overline{\Omega}) = \bar u(\overline{\Omega}) = \frac{3}{4}\pi$. The optimal state is defined as solution of
    \begin{equation*}
        \bar y \in H^1_0(\Omega), \quad 
        \int_\Omega \nabla \bar y \cdot \nabla v \,\d\lambda^2 = \frac{3}{4} \int_{\partial B_{1/2}(0)} v\, \d\HH^1
        \quad \forall\, v\in H^1_0(\Omega) .
    \end{equation*}
    For the adjoint state, we choose 
    \begin{equation*}
        p(\xi) := \|\xi\|^2 - 1
    \end{equation*}        
    so that $p = 0$ on $\partial \Omega$. Note that $p$ obviously satisfies the smoothness assumptions in \cref{thm:metrichausdorff}.
    In order to fulfill the adjoint equation \eqref{eq:adjpde}, we define the desired state $y_d$ by 
    \begin{equation*}
        y_d := \bar y + \laplace p \in H^1(\Omega).
    \end{equation*}
    Note that $\{ \xi \in \overline{\Omega} \colon \|\nabla p(\xi)\| = 1 \} = \partial B_{1/2}(0) = \supp(\bar u)$ in accordance with the 
    complementarity relation in \eqref{eq:complnablap}. In order to compute the $\overline{c}$-conjugate function of $\psi = - p$, we observe that 
    \begin{equation}\label{eq:transrayex}
        \argmin_{\eta \in \omega_1} \|\eta - x\| - \psi(\eta) 
        = \argmin_{\eta \in \overline{\Omega}} \|\eta - x\| + p(\eta) =
        \begin{cases}
            x, & \|x\| < \frac{1}{2} , \\
             \frac{x}{2 \|x\|}, & \|x\| \geq \frac{1}{2}
        \end{cases}          
    \end{equation}
    and consequently,
    \begin{equation*}
        \psi^{\overline{c}}(x) = 
        \begin{cases}
            p(x), & \|x\| < \frac{1}{2}, \\
            \|x\| - \frac{5}{4}, & \|x\| \geq \frac{1}{2} .
        \end{cases}
    \end{equation*}
    In light of \eqref{eq:transrayex}, we see that the transport rays in this example are given by $[\xi, 2 \xi]$ with $\xi \in \partial B_{1/2}(0)$.
    For $A, B \in \BB(\overline{\Omega})$, we define 
    \begin{equation*}
        \bar\pi(A \times B) := \int_0^{2\pi}  \int_{1/2}^1 \chi_B(\tfrac{1}{2} \cos\varphi, \tfrac{1}{2} \sin\varphi) 
        \chi_A(\varrho \cos\varphi, \varrho \sin\varphi) \varrho\,\d \varrho \,\d\varphi
    \end{equation*}
    so that $\bar\pi \geq 0$,
    \begin{equation*}
        \bar\pi(A \times \overline{\Omega}) 
        = \lambda^2\big[A \cap (B_1(0)\setminus B_{1/2}(0))\big] = u^0(A) ,
    \end{equation*}
    and
    \begin{equation*}
    \begin{aligned}
        \bar\pi(\overline{\Omega} \times B) &= \int_0^{2\pi}  \int_{1/2}^1 \chi_B(\tfrac{1}{2} \cos\varphi, \tfrac{1}{2} \sin\varphi) \varrho\,\d \varrho \,\d\varphi \\
        & = \frac{3}{4} \, \int_0^{2\pi}  \chi_B(\tfrac{1}{2} \cos\varphi, \tfrac{1}{2} \sin\varphi) \frac{1}{2}\, \d\varphi
        = \frac{3}{4} \,\HH^1(\partial B_{1/2}(0)\cap B) = \bar u(B) .
    \end{aligned}
    \end{equation*}
    Therefore, $\bar\pi$ is feasible for the Kantorovich problem associated with $u^0$ and $\bar u$.
    For its support we obtain
    \begin{align*}
        \supp(\bar\pi) 
        & = \big\{ (x, \xi) \in \overline{B_1(0)\setminus B_{1/2}(0)} \times \partial B_{1/2}(0) \colon \tfrac{x}{\|x\|} = \tfrac{\xi}{\|\xi\|} \big\} \\
        & = \big\{ (x, \xi) \in \overline{B_1(0)\setminus B_{1/2}(0)} \times \partial B_{1/2}(0) \colon \|x\| - \tfrac{1}{2} = \|x - \xi\| \big\} \\
        & = \big\{ (x, \xi) \in \supp(u^0) \times \supp(\bar u) \colon \psi^{\overline{c}}(x) + \psi(\xi) = \|x - \xi\| \big\} .
    \end{align*}
    Therefore, $\bar\pi$ and $\psi = - p$ satisfy the complementarity condition of the Kantorovich problem 
    such that $\bar\pi$ is the solution of the Kantorovich problem, while $\psi$ is the dual solution.
    All in all, we have seen that $\bar u$ and $\bar y$ satisfy the state equation, $p$ satisfies the adjoint equation, $\bar\pi$ is the solution of 
    the Kantorovich problem associated with $u^0$ and $\bar u$, and $\psi = - p$ solves the dual Kantorovich problem. 
    Therefore, according to \cref{thm:fon}, $\bar u$ is a solution of the optimal control problem, which is not absolutely continuous
    w.r.t.\ the Lebesgue measure, although $u^0$ is so.
    
    Note moreover that, according to \cref{thm:fon}, the adjoint state $p$ is a solution of the dual Kantorovich problem, but it is not $c$-concave.
    Owing to \cite[Proposition~3.1]{santambrogio}, a function is $c$-concave in case of metric costs, if and only if its Lipschitz constant 
    is less or equal $1$, which is obviously not fulfilled in this example, cf.\ also \cref{rem:pnotcconc} in this context.
\end{example}

The next result shows that no mass is transported at all in case of metric transportation costs, if the Tikhonov parameter is sufficiently large.
In a sense, this corresponds to the well-known sparsity results for $L^1(\Omega)$- or Radon-norm regularization terms, where 
the optimal control vanishes, if the Tikhonov parameter is chosen large enough, cf., e.g., \cite{stadler09, ClasonKunisch2011, CasasKunisch2014}.

\begin{proposition}\label{prop:sparsity}
    Suppose that $d_0 = d_1 = d$ and that $\omega_0 = \omega_1 =: \omega$ is convex and 
    let the transportation costs be given by  $c(x,\xi) = \|x-\xi\|$. Moreover, let $\bar u$ be a locally optimal control with associated state 
    $\bar y$ and adjoint state $p$ and assume that the adjoint state is Lipschitz continuous on $\omega$ with 
    Lipschitz constant satisfying 
    \begin{equation}\label{eq:lipp}
        \lip_\omega(p) < \alpha.
    \end{equation}       
    Then $\bar u = u^0$.
\end{proposition}

\begin{proof}
    We again define $\psi := - \alpha^{-1} p$. Then, by assumption, the Lipschitz constant of $\psi$ is less than one such that 
    there holds $\psi^{\overline{c}} = - \psi$, cf.\ e.g.\ \cite[Proposition~3.1]{santambrogio}.
    We again employ \eqref{eq:optimality_support_2}, which, in this case, reads
    \begin{equation}\label{eq:supppimetr}
        \supp(\bar\pi)
       \subset \{(x,\xi) \in \omega\times \omega  \colon \psi(\xi) - \psi(x) = \|x-\xi\|\},
    \end{equation}
    where, again, $\bar\pi$ is a solution to the Kantorovich problem associated with $u^0$ and $\bar u$.
    However, since $\lip_\omega(\psi) < 1$ by assumption, there holds $\psi(\xi) - \psi(x) < \|x-\xi\|$ for all $x\neq \xi$ 
    and consequently $\supp(\bar\pi) \subset \{(x, x) \colon x\in \omega\}$.
    Hence, for every Borel set $A\subset \omega$, we obtain
    \begin{equation*}
        u^0(A) = \bar\pi(A\times \omega)
        = \bar\pi(A\times A) = \bar\pi(\omega\times A) = \bar u(A),
    \end{equation*}
    which proves the claim.
\end{proof}

Similarly to the discussion concerning \eqref{eq:kruemmungp} at the end of the \cref{sec:stritctconv}, 
the question arise, in which situations, the condition in \eqref{eq:lipp} is satisfied, which is investigated in the following.

\begin{corollary}\label{cor:u=u^0}
    Let again $d_0 = d_1 = d$ and $\omega_0 = \omega_1 =: \omega$ be convex and the transport costs be given by $c(x,\xi) = \|x-\xi\|$.
    Suppose in addition that $\omega \subset\Omega$.
    Furthermore, assume that $J$ is Fr\'echet differentiable from $W^{1,q}(\Omega)$ to $\R$. 
    Suppose moreover that there exists $r > d$ such that, for every $M > 0$,  there is a constant $C_M < \infty$ such that 
    \begin{equation}\label{eq:nablaJbound}
        \sup_{\|y\|_{W^{1,q}(\Omega)} \leq M} \norm{J'(y)}_{L^r(\Omega)} \leq C_M
        .
    \end{equation}
    Then the number
    \begin{equation}\label{eq:Cpest}
        C_p := \sup_{|u|(\omega) \leq |u^0|(\omega)} \|\nabla S^* J' (S(u)) \|_{L^\infty(\omega; \R^d)}
    \end{equation}
    is finite. If the Tikhonov parameter satisfies $\alpha > C_p$, then $\bar u = u^0$ is the unique optimal solution of \eqref{eq:optctrl}. 
\end{corollary}

\begin{proof}
        From \cref{lem:pdeexist}, we get
        a constant $C > 0$ such that
        \begin{equation*}
            \norm{S(u)}_{W_0^{1,q}(\Omega)}
            \le
            C \norm{u}_{\frakM(\omega)}
            =
            C \abs{u}(\omega)
            \le
            C \abs{u^0}(\omega)
            =: M
        \end{equation*}
        for all $u \in \frakM(\omega)$ with $\abs{u}(\omega) \le \abs{u^0}(\omega)$.
        From \eqref{eq:nablaJbound}, we infer
    \begin{equation*}
        C_J := \sup_{|u|(\omega) \leq |u^0|(\omega)} \|J'(S(u))\|_{L^r(\Omega)} \le C_M.
    \end{equation*}
    According to \cite[Theorem 9.11]{GT01}, the following interior regularity estimate holds true for the solution of the adjoint equation 
    \eqref{eq:adjpde}
    \begin{equation}\label{eq:interiorreg}
        \| p \|_{W^{2,r}(\omega)} \leq C\, \|J' (y)\|_{L^r(\Omega)}
    \end{equation}
    with a constant $C > 0$. Here we used that $\dist(\omega, \partial\Omega) > 0$ by assumption. 
    Thus $S^*$, the solution operator of \eqref{eq:adjpde}, is a linear and bounded operator from 
    $L^r(\Omega)$ to $W^{2,r}(\omega) \embed W^{1,\infty}(\omega)$, where the latter embedding follows from 
    $r> d$ and the convexity of $\omega$. Therefore we obtain 
    \begin{equation*}
        C_p \leq \|\nabla S^*\|_{\LL(L^r(\Omega), L^\infty(\omega; \R^d))} \, C_J < \infty
    \end{equation*}
    as claimed.
    
    To show the second assertion, let $\alpha > C_p$ and let
    $\bar u$ be an arbitrary local minimizer with associated adjoint state 
    $p$ according to \cref{thm:fon}. 
    Then $\lip_\omega(p) = \| \nabla p \|_{L^\infty(\omega)} \leq C_p < \alpha$
    and \cref{prop:sparsity} implies $\bar u = u^0$.
\end{proof}

\begin{remark}\label{rem:tracking}
    The condition in \eqref{eq:nablaJbound} is rather restrictive, in particular if $d > 2$.
    To illustrate this issue, let us again consider the classical example of a tracking type objective, i.e., 
    \begin{equation*}
        J(y) := \tfrac{1}{2} \| y - y_d \|_{L^2(\Omega)}^2
    \end{equation*} 
    with a given desired state $y_d \in L^\infty(\Omega)$. 
    Due to the Sobolev embedding $W^{1,q}(\Omega) \embed L^s(\Omega)$ with $s = dq/(d-q)$, 
    we then obtain $J'(y) = y - y_d\in L^s(\Omega)$. Now, since $q < d/(d-1)$, there holds
    $s < d/(d-2)$. Thus, condition \eqref{eq:nablaJbound} is satisfied for $d=2$
    if $q$ is chosen large enough.
    In case of $d = 3$ however, we arrive at $s<3$ such that, because of $r > 3$, \eqref{eq:nablaJbound} 
    is not automatically fulfilled by the regularity guaranteed by the state equation. 
    There are however relevant examples, where this condition is met even if $d>2$. 
    If we consider for instance the tracking-type objective from \eqref{eq:tracking} with $\dist(\overline{D}, \omega) > 0$, then
    a classical localization argument shows that $y =  S(u) \in C(\overline{D})$, since the right hand side $u$ is contained in $\omega$.
    Consequently $J'(y) = \chi_D (y - y_d) \in L^\infty(\Omega)$, which implies \eqref{eq:nablaJbound} for this example.
\end{remark}

In case of metric transportation costs, one can slightly sharpen the result of \cref{thm:nodiractransport}, as the following result shows:

\begin{proposition}\label{thm:nodirac3d}
    Consider again metric transportation costs, i.e., $c(x,\xi) = \|x-\xi\|$ and let $d_0 = d_1 = 3$.
    Let $J$ be of tracking type, i.e., given by \eqref{eq:trackingOmega}, 
    with a desired state $y_d \in L^s(\Omega)$ with $s > 3$. 
    Then, a point $\xi  \in \interior(\omega_1)$ can only be an atom of an optimal control $\bar u$ if it is also an atom of the prior $u^0$ 
    and no mass is transported to $\xi$ from any other point of $\omega_0$.
    Moreover, in this case, we have $\bar u(\{\xi\}) \le 8 \pi \alpha$. 
\end{proposition}
\begin{proof}
    The first assertion has already been proven in \cref{thm:nodiractransport}. 
    To show the second assertion, let us return to the proof of \cref{thm:nodiractransport}, more precisely to \eqref{eq:ximin}.
    If $\bar u(\xi) = \beta > 0$ and $\xi \in \interior(\omega_1)$, then \cref{thm:nodiractransport} implies $(\xi,\xi) \in \supp(\bar \pi)$ and thus, there must necessarily hold
    \begin{equation}\label{eq:ximin2}
        \xi \in \argmin_{\eta \in \omega_1}
        \Big\{
            \underbrace{\norm{ \eta - \xi} + \frac{1}{\alpha}\, \big(p_1(\eta) + p_2(\eta) + p_3(\eta)\big)}_{\displaystyle{=: f_{\xi}(\eta)}} \Big\} 
    \end{equation}   
    with $p_1$, $p_2$, and $p_3$ as defined in \eqref{eq:defp123}. As seen in the proof of \cref{thm:nodiractransport}, 
    \begin{equation*}
        (p_1 + p_2 + p_3)(\eta) = - \frac{\beta}{8\pi\alpha}\, \|\eta - \xi\| + p_3(\eta) + \underbrace{ \beta\, h(\eta) + p_2(\eta) }_{\displaystyle{=: \varphi(\eta)}} ,
    \end{equation*}
    where $h$ is a harmonic correction of the boundary values, cf.\ \eqref{eq:p13d}, such that $\varphi \in C^{1, \gamma}(B_\rho(\xi))$, provided that $\rho > 0$ 
    is sufficiently small, see also \eqref{eq:p2est}. Moreover, $p_3$ is superharmonic such that, again, \eqref{eq:superharmest} holds, this time with $\kappa = \gamma$, 
    i.e., for every $r < \rho$, there exists $\eta_r \in \partial B_r(\xi)$ with 
    \begin{equation*}
        p_3(\eta_r) + \varphi(\eta_r) \leq p_3(\xi) + \varphi(\xi) +   \|\varphi\|_{C^{1,\gamma}(B_\rho(\xi))} \, r^{1+\gamma} .
    \end{equation*}
    Consequently, similarly to \eqref{eq:fxabstieg}, we obtain
   \begin{equation*}
    \begin{aligned}
        f_{\xi}(\eta_r) & = \|\eta_r - \xi\| - \frac{\beta}{8\pi\alpha} \, \|\eta_r - \xi\|  + \frac{1}{\alpha} \big( p_3(\eta_r) + \varphi(\eta_r)\big) \\
        & \leq f_{\xi}(\xi) +  \frac{1}{\alpha}\, \|\varphi\|_{C^{1,\gamma}(B_\rho(\xi))} \, r^{1+\gamma} - \Big(\frac{\beta}{8\pi\alpha} - 1\Big) r .
    \end{aligned}
    \end{equation*}
    Therefore, if $\bar u(\{\xi\}) = \beta > 8\pi\alpha$ were true,
    this implies $f_\xi(\eta_r) < f_\xi(\xi)$ for $r > 0$ small enough
    which contradicts \eqref{eq:ximin2}.
\end{proof}

\begin{remark}\label{rem:diraccond}
    The above method of proof can be generalized to more general transportation costs of the form $c(x, \xi) = h(\|x - \xi\|)$
    with a function $h:\R \to \R$ that satisfies $h(0) = 0$. 
    To be more precise, in three dimensions, if $\xi \in \interior(\omega_1)$ is an atom of $\bar u$ with mass $\beta >0$, 
    then, for all $r>0$ sufficiently small, there must necessarily hold that
    \begin{equation}\label{eq:nodiracineq}
        h(r) \geq \frac{\beta}{8 \pi \alpha}\,r - \frac{1}{\alpha} \, \|\varphi\|_{C^{1,\gamma}(B_\rho(\xi))}\, r^{1+\gamma} 
        ,
    \end{equation}
    where $\varphi$ is defined as in the above proof.
    If now $c(x, \xi) = \frac{1}{\kappa} \|x - \xi\|^\kappa$ with $\kappa\in (1, 2)$ so that $h(r) = \frac{1}{\kappa} r^\kappa$, 
    then this condition will always be violated, if $r> 0$ is sufficiently small. Thus, in case of power-type transport costs with exponent between $1$ 
    and $2$, any optimal control $\bar u$ cannot possess an atom in $\interior(\omega_1)$ 
    provided that $J$ is of tracking type with a desired state $y_d \in L^s(\Omega)$ with $s > 3$.
    Note that this case is not covered by \cref{cor:nodiracinterior}, since the transport costs are not twice continuously differentiable in this case.\\    
    In the two dimensional case, the situation changes due to the different structure of the fundamental solution, cf.~\eqref{eq:p12d}.  
    The condition analogous to \eqref{eq:nodiracineq} then reads as follows: 
    an optimal control $\bar u$ can only have an atom at $\xi\in \interior(\omega_1)$ with mass $\beta > 0$, if
    \begin{equation*}
        h(r) \geq - r^2\Big( \frac{\beta}{8 \pi \alpha} (\ln(r) + 1) - \frac{1}{\alpha} \, \|\varphi\|_{C^{1,1}(B_\rho(\xi))} \Big) 
    \end{equation*}
    for all $r> 0$ small enough. Since this is fulfilled by $h(r) = \frac{1}{\kappa} r^\kappa$ with $\kappa \in (1,2)$ for $r>0$ sufficiently small, 
    the above argument to exclude atoms in $\interior(\omega_1)$ in the case 
    $c(x, \xi) = \frac{1}{\kappa} \|x - \xi\|^\kappa$, $\kappa\in (1, 2)$, does not apply in two dimensions.
\end{remark}

We end this section with a trivial example that illustrates our above findings on metric costs.

\begin{example}
    Let $d=3$ and $\overline{\Omega} = \omega_0 = \omega_1 = \overline{B_1(0)}$. Furthermore, $J$ is given by the 
    tracking type objective from \eqref{eq:trackingOmega}. Define the adjoint state $p$ as 
    \begin{equation}\label{eq:adjex}
        p(\xi) := - \kappa \, \big( \|\xi\| - 1\big) \quad \text{with} \quad 0 < \kappa < \alpha
    \end{equation}
    so that $p$ satisfies the homogeneous Dirichlet boundary conditions and $- \laplace p = 8\pi\kappa\,\Phi_0$,
    where $\Phi_0(\xi) =  \frac{1}{4 \pi} \|\xi\|^{-1}$ is the fundamental solution. We set $y_d \equiv 0$ and $\bar y = 8\pi\kappa\,\Phi_0$ so that 
    \begin{equation*}
        \bar u = - \laplace \bar y = 8\pi\kappa\,\delta_0 ,
    \end{equation*}
    where $\delta_0$ denotes the Dirac measure at $\xi = 0$. Note that $\bar u(\{0\}) = 8\pi\kappa < 8\pi\alpha$, cf.\ \cref{thm:nodirac3d}.
    We moreover set the prior to $u^0 = 8\pi\kappa\,\delta_0$ such that $\supp(\bar \pi) = \{(0,0)\}$.
    Note that for
    \begin{equation*}
        \psi(\eta) = -\frac{1}{\alpha}\, p(\eta) = \frac{\kappa}{\alpha} \big( \|\eta\| - 1\big) 
    \end{equation*}
    we have $\lip_{\overline{\Omega}}(\psi) = \lip_{\overline{\Omega}}(p)/\alpha = \kappa/\alpha < 1$.
    Consequently,
    \cite[Proposition~3.1]{santambrogio}
    implies $\psi^{\overline{c}} = -\psi$.
    Hence, \cref{thm:fon} is applicable and this shows that $\bar u$
    is indeed a minimizer.
\end{example}

\begin{remark}
    \label{rem:some_random_remark}
    We observe that the metric transportation costs are crucial in the above example. If $c(x, \xi) = \frac{1}{\kappa} \|x - \xi\|^\kappa$ 
    with $\kappa > 1$, then the transport costs cannot compensate for the norm contribution of the adjoint state in \eqref{eq:adjex} and $\xi = 0$ 
    does not solve the minimization problem analogous to \eqref{eq:ximin2}.
    Therefore, the mass of $u^0$ located at $\xi = 0$ is transported away from that point 
    in case of smooth transportation costs. 
    This is in essence the reason, why no Dirac can appear in $\interior(\omega_1)$ 
    in case of smooth transportation costs in three dimensions, cf.\ \cref{cor:nodiracinterior} and \cref{rem:diraccond}. 
\end{remark}


\section*{Acknowledgement}
The authors are very grateful to Christian Clason (KFU Graz) and Paul Manns (TU Dortmund) for several helpful discussions 
in the early state of the manuscript.

\appendix

\section{Subdifferential of the generalized transportation distance}\label{sec:subdifftrans}

In this section, we derive the characterization of the subdifferential of $\wdist^c_{\mu}$ that is used in \cref{thm:fon}. 
We point out that the result is not new and can for instance be found in \cite[Proposition~7.17]{santambrogio}, 
but, for convenience of the reader, we present the proof in detail.
Throughout this section, let $c: \omega_0 \times \omega_1 \to \R$ be a continuous cost functional
on the compact sets $\omega_0 \subset \R^{d_0}$, $\omega_1 \subset \R^{d_1}$,
and $\mu \in \frakM(\omega_0)$ be a given marginal satisfying $\mu \geq 0$.
Let us define the functional $F$ by
\begin{equation}\label{eq:defF}
    F : C(\omega_1) \ni \psi \mapsto - \int_{\omega_0} \psi^{\overline{c}}(x) \,\d \mu(x) \in \R, 
\end{equation}
where we used that the $\overline{c}$-conjugate is bounded due to the compactness of $\omega_1$.

\begin{lemma}\label{lem:Fconj}
    The conjugate functional to $F$ is given by 
    \begin{equation}
        F^* : \frakM(\omega_1) \ni \nu \mapsto \wdist^c_{\mu}(\nu) \in \R\cup \{\infty\}.
    \end{equation}
\end{lemma}

\begin{proof}
    The assertion follows directly from the Kantorovich duality in compact sets. It is straightforward to see that 
    the pre-dual problem to the Kantorovich problem \eqref{eq:kant} is 
    \begin{equation*}
        \sup\set*{
            \int_{\omega_0} \varphi \,\d \mu  + \int_{\omega_1} \psi \,\d \nu \given
            \varphi \in C(\omega_0), \; \psi \in C(\omega_1), \;
            \varphi \oplus \psi \le c
        }, 
    \end{equation*}        
    cf.\ also \cite[Section~1.2]{santambrogio}.
    Here, $(\varphi \oplus \psi)(x, \xi) := \varphi(x) + \psi(\xi)$.
    Since the regularity condition 
    \begin{equation*}
        0 \in \interior\big\{ v_0 \oplus v_1 - c + v \colon v_0 \in C(\omega_0), v_1 \in C(\omega_1), v\in C(\omega_0 \times \omega_1), v \geq 0 \big\}
    \end{equation*}
    is trivially fulfilled
    (e.g., by choosing $v_0 \equiv - \|c\|_{C(\omega_0 \times \omega_1)} - 1$
    and $v_1 \equiv 0$), the duality result from 
    \cite[Theorem~2.187]{BonnansShapiro2000} is applicable, which gives
    \begin{equation}\label{eq:kantdual}
    \begin{aligned}
        \wdist^c_{\mu}(\nu)
        &= \sup\Big\{ 
        \begin{aligned}[t]
            \int_{\omega_0} \varphi \,\d \mu  &+ \int_{\omega_1} \psi \,\d \nu : 
            \varphi \in C(\omega_0), \; \psi \in C(\omega_1), \\[-1.5ex]
            & \qquad\qquad \varphi(x) + \psi(y) \leq c(x,y) \;\forall \, (x,y) \in \omega_0\times \omega_1\Big\}
        \end{aligned} \\
        &= \sup\Big\{ 
        \begin{aligned}[t]
            \int_{\omega_0} \varphi \,\d \mu  + \int_{\omega_1} \psi \,\d \nu : \;
            & \varphi \in C(\omega_0), \; \psi \in C(\omega_1), \\[-1.5ex]
            & \varphi(x) \leq \inf_{y\in \omega_1} c(x,y) - \psi(y) \;\forall \, x\in \omega_0\Big\}
        \end{aligned} \\
        &= \sup\Big\{ 
        \int_{\omega_0} \psi^{\overline{c}} \,\d \mu  + \int_{\omega_1} \psi \,\d \nu :  \psi \in C(\omega_1)\Big\}
        = F^*(\nu),
    \end{aligned}
    \end{equation}
    where we used the continuity of $\psi^{\overline{c}}$ and $\mu \geq 0$ for the second to last equality.
\end{proof}

\begin{lemma}
    The functional $F$ from \eqref{eq:defF} is convex and continuous on the whole $C(\omega_1)$.
\end{lemma}

\begin{proof}
    It is easily seen that the mapping $C(\omega_1) \ni \psi \mapsto \psi^{\overline c}(x)$ is concave 
    for every $x\in \omega_0$ such that the non-negativity of $\mu$ implies that $F$ is convex. 
    Moreover, it holds that 
    \begin{equation*}
        \|\psi_1^{\overline c} - \psi_2^{\overline c}\|_{C(\omega_0)} \leq \|\psi_1 - \psi_2\|_{C(\omega_1)}
        \quad \forall \, \psi_1, \psi_2 \in C(\omega_1),
    \end{equation*}
    which implies the continuity of $F$.
\end{proof}

Consequently, $F$ is proper, convex, and lower semicontinuous such that
for all $\psi \in C(\omega_1)$
and $\nu \in \frakM(\omega_1)$
the equivalences
\begin{equation*}
    \psi \in \partial F^*(\nu)
    \quad \Longleftrightarrow \quad
    \nu \in \partial F(\psi) 
    \quad \Longleftrightarrow \quad
    F(\psi) + F^*(\nu) = \int_{\omega_1} \psi\,\d\nu
\end{equation*}
are obtained by standard arguments of convex analysis.
Thanks to \cref{lem:Fconj}, this means
\begin{equation*}
    \psi \in \partial \wdist^c_{\mu}(\nu) 
    \quad \Longleftrightarrow \quad
    \wdist^c_{\mu}(\nu) = \int_{\omega_0} \psi^{\overline c} \,\d\mu + \int_{\omega_1} \psi\, \d\nu,
\end{equation*}
which, in view of the equivalent reformulation of the dual Kantorovich problem in \eqref{eq:kantdual} yields
the following

\begin{proposition}\label{prop:kantsubdiff}
    A function $\psi \in C(\omega_1)$ is an element of the convex subdifferential $\partial\wdist^c_{\mu}(\nu)$, 
    if and only if $(\psi^{\overline{c}}, \psi)$ solves the pre-dual Kantorovich problem with marginals $\mu$ and $\nu$.
\end{proposition}


\section{Miscellaneous auxiliary results}\label{sec:misc}

\begin{lemma}\label{lem:localreg}
    Let $\Omega \subset \R^d$, $d \in \set{2, 3}$, be a bounded Lipschitz domain
    and fix $q \in (p_\Omega', \frac{d}{d-1})$, where $p_\Omega > d$ denotes again the exponent from \cite[Theorem~0.5(a)]{JK95}.
    Then, for every $\mu \in \frakM(\Omega)$, there exists a unique solution $w \in W^{1,q}_0(\Omega)$ of 
    \begin{equation}\label{eq:jerisonkenig}
        - \laplace w = \mu \quad \text{in } W^{-1,q}(\Omega),
    \end{equation}
    and, for every open set $M \subset \Omega$ with $\dist(M, \supp(\mu)) > 0$, there holds $w \in C(\overline{M})$.
\end{lemma}

\begin{proof}
    The existence and uniqueness of solutions to \eqref{eq:jerisonkenig} is again due to \cite[Theorem~0.5]{JK95}.
    The proof of the continuity on $\overline{M}$ is based on a classical localization and boot strapping argument. 
    Due to $\dist(M, \supp(\mu)) > 0$,  there exists an open set $N$ such that 
    \begin{equation*}
      M \subset\subset N \subset\subset \R^d \setminus \supp(\mu).
    \end{equation*}
    Then we have $\dist(N, \supp(\mu)) > 0$ such that 
    there exists a non-negative function $\varphi \in C^\infty(\R^d)$ 
    with $\varphi \equiv 1$ in $N$ and $\varphi \equiv 0$ in $\supp(\mu)$. For every $v\in W^{1,q'}_0(\Omega)$, we obtain
    \begin{equation*}
    \begin{aligned}
        \int_\Omega \nabla (\varphi \, w) \cdot \nabla v\, \d\lambda^d 
        & = \int_\Omega w\, \nabla \varphi \cdot \nabla v\, \d\lambda^d + \int_\Omega \nabla w \cdot \nabla (\varphi \, v)\, \d\lambda^d 
        - \int_\Omega (\nabla w \cdot \nabla \varphi) v\, \d\lambda^d \\
        & = \int_\Omega w\, \nabla \varphi \cdot \nabla v\, \d\lambda^d + \underbrace{\int_\Omega \varphi\, v \,\d\mu}_{=0} 
        - \int_\Omega (\nabla w \cdot \nabla \varphi) v\, \d\lambda^d,
    \end{aligned}
    \end{equation*}
    where we used that $\varphi \equiv 0$ in $\supp(\mu)$.
    Therefore, $z \coloneqq w\varphi \in W^{1,q}_0(\Omega)$ solves 
    \begin{equation}\label{eq:zeq}
        - \laplace z = g \quad \text{in } W^{-1,q}(\Omega)
    \end{equation}
    with $g \in W^{-1,q}(\Omega)$ defined as
    \begin{equation*}
        \dual{g}{v} \coloneqq \int_\Omega w\, \nabla \varphi \cdot \nabla v\, \d\lambda^d - \int_\Omega (\nabla w \cdot \nabla \varphi) v\, \d\lambda^d, \
        \quad v \in W^{1,q'}_0(\Omega) .
    \end{equation*}
    Now, since $w \in W^{1,q}_0(\Omega) \embed L^s(\Omega)$ with $s = dq/(d-q)$ and $\varphi \in C^\infty(\R^d)$, 
    the right hand side $g$ is also well defined as a functional on $W^{1,\kappa}_0(\Omega)$ for all $\kappa$ satisfying 
    \begin{equation*}
        \kappa \geq s' = \frac{s}{s-1} = \frac{dq}{dq - d + q} = \frac{dq}{(d+1)q - d}.
    \end{equation*}    
    In case of $d = 2$, the condition $q < d/(d-1)$ implies $s' > 1$ such that every $\kappa > 1$ is allowed. If we choose $\kappa < 2$, but sufficiently close to two, 
    then \cite[Theorem~0.5]{JK95} yields that \eqref{eq:zeq} considered as an equation in $W^{-1,\kappa'}(\Omega)$
    admits a unique solution $\hat z \in W^{1,\kappa'}_0(\Omega)$ with $\kappa' > 2$.
    Since $\kappa' > q'$, $\hat z$ also solves \eqref{eq:zeq} and, as this equation is uniquely solvable by \cite[Theorem~0.5]{JK95}, we obtain 
    $z = \hat z \in W^{1,\kappa'}_0(\Omega) \embed C(\overline{\Omega})$ by means of Sobolev embeddings.
    Thus, we have $w \varphi = z \in C(\overline{\Omega})$ and, due to $\varphi \equiv 1$ in $N$, this gives the result for $d=2$.
    
    In case of $d=3$, we obtain $s' > 3/2$ and thus $\kappa > 3/2$, which is unfortunately not sufficient to deduce the result, so we have to repeat the argument. 
    We set $\kappa = 2$ and thus obtain $z\in H^1_0(\Omega)$ as solution of \eqref{eq:zeq} by the Lax-Milgram lemma.
    Now, we choose a function $\phi \in C^\infty(\R^d)$ with $\phi \equiv 1$ in $M$ and $\phi \equiv 0$ in $\overline{\Omega} \setminus N$.
    Then, for all $v\in H^1_0(\Omega)$, there holds
    \begin{equation*}
    \begin{aligned}
        \int_\Omega \nabla (\phi \, z) \cdot \nabla v\, \d\lambda^d 
        & = \int_\Omega z\, \nabla \phi \cdot \nabla v\, \d\lambda^d + \int_\Omega \nabla z \cdot \nabla (\phi \, v)\, \d\lambda^d 
        - \int_\Omega (\nabla z \cdot \nabla \phi) v\, \d\lambda^d \\
        & = \int_\Omega z\, \nabla \phi \cdot \nabla v\, \d\lambda^d + \underbrace{\dual{g}{\phi\, v}}_{=0} 
        - \int_\Omega (\nabla z \cdot \nabla \phi) v\, \d\lambda^d,
    \end{aligned}
    \end{equation*}
    where we used that $\phi$ vanishes in $\Omega \setminus N$ and $\varphi$ is constant in $N$. Similar to before, the right hand side
    \begin{equation*}
        \tilde g : H^1_0(\Omega) \to \R, \quad 
        \dual{\tilde g}{v} \coloneqq \int_\Omega z\, \nabla \phi \cdot \nabla v\, \d\lambda^d - \int_\Omega (\nabla z \cdot \nabla \phi) v\, \d\lambda^d
    \end{equation*}
    is well defined on $W^{1, \beta}_0(\Omega)$ with $\beta \geq 6/5$,
    since $z \in H^1_0(\Omega) \embed L^6(\Omega)$ in case of $d=3$ and $\phi \in C^\infty(\R^d)$.
    Therefore, we may choose $\beta$ smaller than, but sufficiently close to $3/2$. Then, according to \cite[Theorem~0.5]{JK95}, 
    the Poisson equation 
    \begin{equation*}
        - \laplace \tilde z = \tilde g \quad \text{in } W^{-1, \beta'}(\Omega)
    \end{equation*}
    admits a unique solution $\tilde z \in W^{1,\beta'}_0(\Omega)$. By construction, $\phi z$ is the unique solution of this equation 
    considered as an equation in $H^{-1}(\Omega)$, 
    and thus $\phi z = \tilde z \in W^{1,\beta'}_0(\Omega) \embed C(\overline{\Omega})$ because of $\beta' > 3$. 
    Since $\phi = \varphi \equiv 1$ on $M$, we finally obtain $w \in C(\overline{M})$ as claimed.
\end{proof}

\begin{lemma}[Comparison principle]\label{lem:comparison}
    Let $\Omega \subset \R^d$, $d \in \set{2, 3}$, be a bounded Lipschitz domain and
    $\mu \in \frakM(\Omega)$ be given with $\mu \geq 0$. Then $w = (-\laplace)^{-1} E_\Omega^* \mu \geq 0$ a.e.\ in $\Omega$.
\end{lemma}

\begin{proof}
    We argue by duality. For that purpose, let $g \in C^\infty_c(\Omega)$, $g\geq 0$ a.e.\ in $\Omega$, be arbitrary and 
    define $v\in W^{1,q'}_0(\Omega) \embed C(\overline{\Omega})$ as solution of  
    \begin{equation*}
        \int_\Omega \nabla v \cdot \nabla \varphi\,\d\lambda^d = \int_\Omega g\,\varphi\,\d\lambda^d
        \quad \forall\, \varphi \in W^{1,q}_0(\Omega),
    \end{equation*}
    where $q < d / (d-1)$ is chosen such that, according to
    \cite[Theorem~0.5]{JK95}, this equation admits a unique solution.
    Moreover, by testing this equation with 
    $\varphi = \min\{0, v\}$ and using the continuity of $v$, we immediately verify $v(x) \geq 0$ for all $x\in \overline{\Omega}$.
    Together with the non-negativity of $\mu$, this yields
    \begin{equation*}
        \int_\Omega g\, w\, \d\lambda^d = \int_\Omega \nabla v \cdot \nabla w\,\d\lambda^d 
        = \int_\Omega v \,\d \mu \geq 0.
    \end{equation*}
    Since $g \in C^\infty_c(\Omega)$, $g\geq 0$, was arbitrary, this gives the claim.
\end{proof}

The next lemma provides a formula for the support of the push-forward of a measure.
Note that it is similar to \cite[Proposition~2.14]{Maggi2012},
which is not directly applicable since $T$ is only defined on $X$.

\begin{lemma}
    \label{lem:support_pushforward}
    Let $X \subset \R^n$ be compact.
    Further, let $\mu$ be a finite Borel measure on $X$
    and let $T \colon X \to \R^m$ be continuous.
    Then, $\supp(T_{\#} \mu) = T(\supp(\mu))$.
\end{lemma}
\begin{proof}
    For convenience, we set $\nu := T_{\#} \mu$.

    Let $x \in \supp(\mu)$ be arbitrary. Since $T$ is continuous,
    for all $r > 0$, there exists $\rho > 0$ such that $T(B_\rho(x)) \subset B_r(T(x))$. It follows that
    \begin{equation*}
        0 < \mu(B_\rho(x)) \leq \mu\big(T^{-1}(B_r(T(x)))\big) = \nu(B_r(T(x)))\quad \forall \, r > 0,
    \end{equation*}
    which shows that $T(\supp(\mu)) \subset \supp(\nu)$.
    
    For the reverse inclusion, let $\xi\in \supp(\nu)$ be arbitrary.
    From $\nu = T_{\#} \mu$ we get
    $\mu(T^{-1}(B_r(\xi))) > 0$ for all $r > 0$. 
    Thus, for every $r>0$, \eqref{eq:prop_support}
    implies the existence of $x_r \in T^{-1}(B_r(\xi)) \cap \supp(\mu)$. By compactness of $X$, we obtain 
    the existence of $x\in X$ such that $x_r \to x$ as $r\searrow 0$ (at least for a subsequence) and the 
    continuity of $T$ implies $\xi = T(x)$. As $\supp(\mu)$ is closed, we also have $x\in \supp(\mu)$ so that 
    $\xi \in T(\supp(\mu))$.
\end{proof}

\begin{lemma}\label{lem:pisupp}
    Let $\omega_0 \subset \R^{d_0}$ and $\omega_1 \subset \R^{d_1}$ be compact.
    Let $\mu \in \frakM(\omega_0)$ and $\nu \in \frakM(\omega_1)$ be two marginals satisfying $\mu, \nu \geq 0$ 
    and $\mu(\omega_0) = \nu(\omega_1)$. Suppose moreover that $\pi$ is feasible for the associated Kantorovich problem, 
    i.e., $\pi \geq 0$, ${P_0}_\# \pi = \mu$, and ${P_1}_\#\pi = \nu$. 
    Then, for every $x\in \supp(\mu)$ there is a $y\in \supp(\nu)$ such that $(x,y)\in \supp(\pi)$. 
\end{lemma}

\begin{proof}
    Since $\omega_0 \times \omega_1$ is compact
    and since the projections are continuous,
    we can apply
    \cref{lem:support_pushforward} to obtain
    \begin{equation*}
        \supp(\mu) = P_0( \supp(\pi) )
        \qquad\text{and}\qquad
        \supp(\nu) = P_1( \supp(\pi) )
        .
    \end{equation*}
    Let $x \in \supp(\mu)$ be given.
    The first identity shows the existence of $y$ with $(x,y) \in \supp(\pi)$
    and the second identity implies $y \in \supp(\nu)$.
\end{proof}


\printbibliography

\end{document}